\documentclass[11pt,reqno]{amsart}

\usepackage{amsmath}
\usepackage{amsthm}
\usepackage{amssymb}
\usepackage{amsfonts}
\usepackage{amsxtra}
\usepackage[all]{xypic}
\usepackage{fullpage}
\usepackage{amscd}
\usepackage{mathrsfs}
\usepackage[mathscr]{eucal}
\usepackage{tikz}
\usepackage{arydshln}
\usepackage{latexsym}

\let\nc\newcommand
\let\renc\renewcommand

\theoremstyle{plain}
\newtheorem{thm}{Theorem}
\newtheorem{prop}[thm]{Proposition}
\newtheorem{cor}[thm]{Corollary}
\newtheorem{lem}[thm]{Lemma}
\newtheorem{conjecture}[thm]{Conjecture}
\newtheorem{conj}[thm]{Conjecture}
\theoremstyle{definition}
\newtheorem{defn}[thm]{Definition}
\newtheorem{example}[thm]{Example}
\newtheorem{remark}[thm]{Remark}
\newtheorem{claim}[thm]{Claim}
\newtheorem{assumption}[thm]{Assumption}
\numberwithin{thm}{section}

\let\nc\newcommand
\let\renc\renewcommand

\nc{\bthm}{\begin{thm}}
\nc{\ethm}{\end{thm}}
\nc{\blem}{\begin{lem}}
\nc{\elem}{\end{lem}}
\nc{\bcor}{\begin{cor}}
\nc{\ecor}{\end{cor}}
\newcommand{\beq}{\begin{equation}\label}
\nc{\eeq}{\end{equation}}
\nc{\bprop}{\begin{prop}}
\nc{\eprop}{\end{prop}}
\nc{\bdefn}{\begin{defn}}
\nc{\edefn}{\end{defn}}

\numberwithin{equation}{subsection}

\newcommand{\Lmod}[1]{#1\text{-}{\mathsf{mod}}}

\newcommand{\idot}{{\:\raisebox{2pt}{\text{\circle*{1.5}}}}}

\DeclareMathOperator{\Ker}{\mathrm{Ker}}

\DeclareMathOperator{\Ad}{\mathrm{Ad}}

\DeclareMathOperator{\ann}{\mathtt{Ann}}

\DeclareMathOperator{\codim}{\mathrm{codim}}

\DeclareMathOperator{\Spec}{\mathrm{Spec}}

\DeclareMathOperator{\Hom}{\mathrm{Hom}}

\DeclareMathOperator{\GL}{\mathrm{GL}}

\renewcommand{\o}{\otimes }

\nc{\Z}{\mathbb{Z}}
\newcommand{\N}{\mathbb{N}}

\newcommand{\C}{\mathbb{C}}
\newcommand{\h}{\mathfrak{h}}

\nc{\rank}{\mathrm{rank} \,}
\nc{\ds}{\dots}

\let\mc\mathcal

\let\mf\mathfrak

\nc{\loccit}{{\textit{loc. cit. }}}

\nc{\HW}{\bar{H}_{\mathbf{c}}(W)}
\nc{\HK}{\bar{H}_{\mathbf{c}}(K)}
\nc{\HtK}{\widetilde{H}_{\mathbf{c}}(K)}
\nc{\CMW}{\mathsf{CM}_{\mbf{c}}(W)}
\nc{\CMK}{\mathsf{CM}_{\mbf{c}}(K)}
\nc{\mbf}{\mathbf}
\nc{\LK}{\mathsf{Irr}(K)}
\nc{\LW}{\mathsf{Irr}(W)}
\nc{\Res}{\mathsf{Res} \, }
\nc{\Ind}{\mathsf{Ind} \, }

\nc{\cont}{\mathrm{cont}}

\nc{\eWb}{\mathbf{e}_{W_b}}
\nc{\eW}{\mathbf{e}_{W}}
\nc{\msf}{\mathsf}
\nc{\Ui}{\mc{U}_{i,+}}
\nc{\Uone}{\mc{U}_{1,+}}
\nc{\Utwo}{\mc{U}_{2,+}}
\nc{\pa}{\partial}
\nc{\Irr}{\mathsf{Irr}}
\nc{\sm}{\mathrm{sm}}
\nc{\Hilb}{\mathrm{Hilb}}
\nc{\<}{\langle}
\renewcommand{\>}{\rangle}

\nc{\bVerma}{\overline{\Delta}}
\nc{\Verma}{\Delta}
\nc{\Fix}{\mathrm{Fix}}
\nc{\bpr}{\begin{proof}}
\nc{\epr}{\end{proof}}
\nc{\s}{\mathfrak{S}}
\nc{\sgn}[1]{\mathrm{sgn}(#1)}
\nc{\g}{\mathfrak{g}}
\nc{\sing}{\mathrm{sing}}
\nc{\Mat}{\mathrm{Mat}}
\nc{\whh}{\widehat{\C[\h^*]}}
\nc{\whH}{\widehat{H}}
\nc{\wh}[1]{\widehat{#1}}
\nc{\Coh}{\mathsf{Coh}}
\nc{\Be}{\mathcal{B}}
\nc{\id}{\mathrm{id}}
\nc{\opp}{\mathrm{opp}}
\nc{\rdet}{\mathrm{rdet}}
\nc{\bigD}{\widetilde{\mathscr{D}}}
\nc{\CM}{CM}
\nc{\CMM}{\overline{CM}}
\nc{\CMMB}{\widetilde{CM}}
\nc{\PGL}{\mathrm{PGL}}
\nc{\tx}{\tilde{x}}
\nc{\hra}{\hookrightarrow}
\nc{\rightsim}{\stackrel{\sim}{\longrightarrow}}
\nc{\ra}{\rightarrow}
\nc{\red}{\mathrm{red}}
\renc{\t}{\mf{t}}
\nc{\dH}{\mathcal{H}}
\nc{\tit}{\textit}
\nc{\A}{\mathbb{A}}
\nc{\Grel}{\mc{G}^{\mathrm{rel}}}
\nc{\Grat}{\mc{G}^{\mathrm{rat}}}
\nc{\tGrat}{\widetilde{\mc{G}}^{\mathrm{rat}}}
\nc{\Squo}[1]{\A^{(#1)}}
\nc{\twist}{\mathrm{twist}}
\nc{\Cd}{\mc{C}}
\nc{\Span}{\mathrm{Span} \ }
\nc{\Grass}{\mathrm{Gr}}
\nc{\Grad}{\mathrm{G}^{\mathrm{ad}}}
\nc{\Graad}{\mathrm{G}^{\mathrm{Ad}}}
\nc{\lab}{\label}
\nc{\Supp}{\mathrm{Supp}}
\nc{\Wr}{\mathrm{Wr}}
\nc{\QE}{\mathcal{Q}\mathcal{E}}
\renc{\H}{H}
\nc{\ZH}{Z}
\nc{\cX}{X}
\nc{\mb}{\mathbf}
\nc{\DVerma}{\nabla}
\nc{\diag}{\mathrm{diag} \ }
\nc{\Cs}{\C^{\times}}
\nc{\reg}{\mathrm{reg}}
\nc{\bp}{\boldsymbol{\mathsf{p}}}

\nc{\Quasi}{\mathcal{Q}}
\nc{\ueta}{\underline{\eta}}
\nc{\uy}{\mbf{y}}
\nc{\ux}{\mbf{x}}
\nc{\uz}{\mbf{z}}
\nc{\ur}{\mbf{r}}
\nc{\op}{\mathrm{op}}
\nc{\Fo}{\msf{F}}
\nc{\Bi}{\msf{B}}
\nc{\tQGr}{\widetilde{\mathcal{Q}\Grass}}
\nc{\QGr}{\mathcal{Q}\Grass}
\nc{\QErel}{\QE^{\mathrm{rel}}}
\nc{\blambda}{\boldsymbol{\lambda}}
\nc{\bmu}{\boldsymbol{\mu}}
\nc{\ba}{\boldsymbol{a}}
\nc{\bb}{\boldsymbol{b}}
\nc{\bx}{\boldsymbol{x}}
\nc{\bOmega}{\boldsymbol{\Omega}}
\nc{\bmho}{\boldsymbol{\mho}}
\nc{\OmCh}{\Omega}
\nc{\OmCM}{\Omega^{\mathrm{cm}}}
\nc{\OmGrad}{\Omega^{\mathrm{ad}}}
\nc{\OmGr}{\Omega^{\mathrm{qe}}}
\nc{\mOCh}{\mho}
\nc{\mOCM}{\mho^{\mathrm{cm}}}
\nc{\mOGrad}{\mho^{\mathrm{ad}}}
\nc{\mOGr}{\mho^{\mathrm{qe}}}
\nc{\NN}{\mathsf{N}}
\nc{\Der}{\mathrm{Der}}

\begin{document}
\title{Rational Cherednik algebras and Schubert cells}

\author{Gwyn Bellamy}
\address{School of Mathematics and Statistics, University of Glasgow, University Gardens, Glasgow G12 8QW}
\email{gwyn.bellamy@glasgow.ac.uk}

\begin{abstract} 
The representation theory of rational Cherednik algebras of type $\mbf{A}$ at $t = 0$ gives rise, by considering supports, to a natural family of smooth Lagrangian subvarieties  of the Calogero-Moser space. The goal of this article is to make precise the relationship between these Lagrangians and Schubert cells in the adelic Grassmannian. In order to do this we show that the isomorphism, as constructed by Etingof and Ginzburg, from the spectrum of the centre of the rational Cherednik algebra to the Calogero-Moser space is compatible with the factorization property of both of these spaces. As a consequence, the space of homomorphisms between certain representations of the rational Cherednik algebra can be identified with functions on the intersection Schubert cells.  
\end{abstract}
\maketitle

\normalsize

\section{Introduction}

In this article we explore certain aspects of the close relationship between rational Cherednik algebras and the Calogero-Moser integrable system. It was shown in the original paper \cite{EG}, where rational Cherednik algebras were first defined by Etingof and Ginzburg, that the centre of the rational Cherednik algebra of type $\mbf{A}$, at $t = 0$, is isomorphic to the coordinate ring of Wilson's completion of the Calogero-Moser phase space. The Calogero-Moser space is also closely related to the adelic Grassmannian and rational solutions of the KdV hierarchy. As such, natural objects of study of the KdV hierarchy such as the $\tau$ and Baker functions and Schubert cells appear naturally in the setting of the Calogero-Moser space. The purpose of this article is to try and understand how these objects manifest themselves in terms of the representation theory of rational Cherednik algebras. 

In the remainder of the introduction, we outline the main results of the paper.

\subsection{The rational Cherednik algebra}

The  rational Cherednik algebra $H_n$ associated to the symmetric group $\s_n$ at $t = 0$ and non-zero $\mbf{c}$ is a finite module over its centre $Z_n$. The spectrum $\cX_n$ of the affine domain $Z_n$ is a symplectic manifold. For each $p \in \h^*$ and $\blambda \in \Irr (\s_p)$, there exists a natural induced $H_n$-module $\Delta(p,\blambda)$, a \textit{Verma module}. In \cite{End} it was shown that the support  $\Omega_{\bb,\blambda}$, depending only on the image $\bb$ of $p$ in $\h^* / \s_n$, of these Verma modules is a smooth Lagrangian subvariety $\cX_n$. It is these Lagrangians that we aim to study in this paper.

\subsection{The Calogero-Moser space}

Wilson's completion of the Calogero-Moser space can be described as follows, see section \ref{sec:maincm} for details. Let $\CMM_n$ be the set of all pairs of $n \times n$, complex matrices $(X,Y)$ such that the rank of $[X,Y] + I_n$ is one. The group $\PGL_n(\C)$ acts on the space $\CMM_n$ and the Calogero-Moser space $\CM_n$ is defined to be the categorical quotient $\CMM_n /\!/ \PGL_n$. It is a smooth, $2n$-dimensional affine variety. As noted above, Etingof and Ginzburg constructed an isomorphism $\cX_n \rightsim \CM_n$. On the other hand, Wilson showed that the union over all $n$ of the Calogero-Moser spaces $\CM_n$ can be identified with a certain infinite dimensional space, the \textit{adelic Grassmannian}. Thus, there is an embedding of the space $\cX_n$ into this adelic Grassmannian. In order to be able to describe the image of the subspaces $\OmCh_{\bb,\blambda}$ in the adelic Grassmannian, it is more convenient to describe the adelic Grassmannian in terms of certain spaces of quasi-exponentials. 

A holomorphic function $f$ on the complex plane that can be expressed as  
$$
f(x) = e^{b_1 x} g_1(x) + \ds + e^{b_k x} g_k(x),
$$
where $b_i \in \C$ and $g_i(x)$ is a polynomial, is called a quasi-exponential function. Let $\Quasi$ denote the space of all quasi-exponential functions. A finite dimensional subspace $C$ of $\Quasi$ is called \textit{homogeneous} if $C = \bigoplus_{b \in \C} C_b$, where $C_b$ is spanned by functions of the form $e^{b x} g(x)$ for some polynomial $g(x)$. The set of all finite dimensional, homogeneous subspaces of $\Quasi$ is denoted $\QGr$. Using the Wronskian, one can pick out certain distinguished spaces in $\QGr$ called \textit{canonical} spaces. The set of all canonical spaces is denoted $\QE$. Wilson showed that each point in the Calogero-Moser space $\CM_n$ can labeled by a canonical space $C \in \QE$. As a consequence we have bijections
$$
\xymatrix{
\cX_n \ar[r]^{\sim} \ar@/_1pc/[rr]_{\nu_n} & \CM_n \ar[r]^{\sim} & \QE_n
}
$$
where $\QE_n$ is the set of all $n$-dimensional spaces in $\QE$. One of the main goals of this paper is to describe the image of the Lagrangians $\OmCh_{\bb,\blambda}$ under $\nu_n$. 

For $\bb = \sum_{i =1}^k n_i b_i \in \h^* / \s_n$, we define 
$$
\Grass_{\bb}(\QGr) = \Grass_{n_1} \left( e^{b_1 x} \C[x]_{2 n_1} \right) \times \cdots \times \Grass_{n_k} \left( e^{b_k x} \C[x]_{2 n_k} \right),
$$
a projective subvariety of $\QGr$, where $\C[x]_{k}$ is the space of all polynomials of degree less that $k$ and $\Grass_{n} \left( e^{b x} \C[x]_{2 n} \right)$ is the Grassmannian of $n$-dimensional planes in $e^{b x} \C[x]_{2 n}$. Each Grassmannian $\Grass_{n_i} \left(e^{b_i x} \C[x]_{2 n_i} \right)$ has a natural stratification by Schubert cells $\OmGr_{b_i, \lambda^{(i)}}$, which are labeled by all those partitions $\lambda^{(i)}$ that fit into a square with sides of length $2n_i$; see section \ref{sec:cells} for the precise definition. Thus, if $\blambda = (\lambda^{(1)}, \ds, \lambda^{(k)})$ with $\lambda^{(i)} \vdash n_i$, then 
$$
\OmGr_{\bb, \blambda} := \OmGr_{b_1, \lambda^{(1)}} \times \cdots \times \OmGr_{b_k, \lambda^{(k)}}
$$
is a locally closed subvariety of $\Grass_{\bb}(\QGr)$. 

\begin{thm}\label{thm:main}
Let $p \in \h^*$ and $\blambda$ an irreducible $\s_p$-module. Then, the map $\nu_n$ restricts to an isomorphism \textit{of varieties} 
$$
\nu_n : \OmCh_{\bb,\blambda} \rightsim \OmGr_{\bb,\blambda^t},
$$
where $\bb$ is the image of $p$ in $\h^* / \s_n$, and $\blambda^t$ denotes componentwise transpose. 
\end{thm}

The proof of Theorem \ref{thm:main} is given in section \ref{subsection:Schuberintersect}. The essential fact that we shall repeatedly use in the proof of Theorem \ref{thm:main} is that each of the spaces $\cX_n$, $\CM_n$ and $\QE_n$ satisfies a certain factorization property. Namely, there is a map from each of the spaces to $\h^* / \s_n$, 
$$
\xymatrix{
\cX_n \ar[r]^{\sim} \ar[dr]_{\pi} & \CM_n \ar[d]^{\pi^{\mathsf{cm}}} \ar[r]^{\sim} & \QE_n \ar[dl]^{\Supp} \\
 & \h^* / \s_n & 
}
$$ 
such that the fiber of each map over $\bb$ can be factorized as the product of the fibers over $n_i \cdot b_i$, where $i$ runs over $1, \ds, k$, for instance  
$$
\pi^{-1}(\bb) \simeq \pi^{-1}( n_1 \cdot b_1) \times \cdots \times \pi^{-1}(n_k \cdot b_k),
$$
where $\pi^{-1}(n_i \cdot b_i)$ is a closed subvariety of $\cX_{n_i}$. The key step in our work is to show that each of the isomorphisms $\cX_n \rightsim \CM_n$ and $\CM_n \rightsim \QE_n$ is compatible with these factorizations, in the obvious sense. A closely related result \cite{KZChar} appeared whilst this paper was in preparation. The second key fact that we shall repeatedly use is that each of the spaces $\cX_n, \CM_n$ and $\QE_n$ is equipped with a canonical $\Cs$-action such that the isomorphisms between them is $\Cs$-equivariant. 

\subsection{} Dual to the space $\OmCh_{\bb,\blambda}$ is a space $\mOCh_{\ba,\bmu}$, where $\ba = \sum_{i = 1}^l n_i a_i \in \h / \s_n$. It is the support of a dual Verma module, $\nabla (q, \bmu)$, where $q \in \h$ with $\bar{q} = \ba$ and $\bmu = (\mu^{(1)}, \ds, \mu^{(l)})$ is an irreducible $\s_q$-module. We show, Theorem \ref{thm:bigdiagram}, that the image of $\mOCh_{\ba,\bmu}$ under the map $\nu_n$ is the set $\mOGr_{\ba,\bmu}$ of all spaces $C$ of quasi-exponentials in $\QE$ such that the singularities of $C$, counted with multiplicity, are encoded by $\ba$, and $\bmu$ encodes the exponents of $C$ at each singular point. See definition \ref{defn:exponents} for details.

The space $\Hom_{\H_n}(\nabla(q,\bmu),\Delta(p,\blambda))$ is a $\ZH_n$-module, supported on the intersection of $\OmCh_{\bb,\blambda}$ and $\mOCh_{\ba,\bmu}$. We consider the case $p = 0$ so that $\nu_n(\OmCh_{0,\lambda} \cap \mOCh_{\ba,\bmu})$ is contained in $\Grass_n(\C[x]_{2n})$. Set-theoretically, the intersection $\Grass_n(\C[x]_{2n}) \cap \nu_n(\mOCh_{\ba,\bmu})$ is the intersection 
$$
\Omega_{\bmu}(q) = \Omega_{\mu^{(1)}}(q_1) \cap \ds \cap \Omega_{\mu^{(k)}}(q_k)
$$
of a certain collection of Schubert cells in $\Grass_n(\C[x]_{2n})$, where the numbers $q_i$ are specifying complete flags in $\C[x]_{2n}$. Under the assumption\footnote{See assumption \ref{assume:schubert} for details.} that we have an equality of (non-reduced) subschemes
$$
\Grass_n(\C[x]_{2n}) \cap \nu_n(\mOCh_{\ba,\bmu}) = \Omega_{\bmu}(q)
$$
of $\Grass_n(\C[x]_{2n})$, we show in Theorem \ref{thm:intersectassume} that  

\begin{cor}\label{cor:intersect}
We have an isomorphism of zero-dimensional, \textit{Gorenstein} schemes
$$
\nu_n : \OmCh_{0,\lambda} \cap \mOCh_{\ba,\bmu} \rightsim \OmGr_{0,\lambda^t} \cap \Omega_{\bmu}(q)
$$
such that $\Hom_{\H_n}(\nabla(q,\bmu,\mbf{0}),\Delta(0,\blambda,\ba))$ is the coregular ($\simeq$ regular) representation of $\C[\OmCh_{0,\lambda} \cap \mOCh_{\ba,\bmu}]$. 
\end{cor}

We show, independent of the assumption, that 
$$
\dim \Hom_{\H_n}(\nabla(q,\bmu,\mbf{0}),\Delta(0,\lambda,\ba)) = \dim \C[\OmGr_{0,\lambda^t} \cap \Omega_{\bmu}(q)] = \langle \sigma_{\lambda^t}, \sigma_{\mu^{(1)}} \cdots \sigma_{\mu^{(k)}} \rangle, 
$$
where $\sigma_{\idot}$ is the cohomology class in $H^*(\Grass_n(\C[x]_{2n}))$ defined by the closure of a given cell and $\langle - , - \rangle$ is the usual pairing on $H^*(\Grass_n(\C[x]_{2n}))$. 

\subsection{} The results of this article were motivated by recent work of Mukhin, Tarasov and Varchenko. They showed in \cite{BetheCherednik}, that there is an intriguing relationship between the rational Cherednik algebra and the Bethe algebra associated to the Gaudin integrable system. Many of the results of this paper were inspired by analogous results of Mukhin, Tarasov and Varchenko about the representation theory of the Bethe algebra. 

\subsection{Acknowledgements}

The author would like to express his sincerest thanks to Iain Gordon and Victor Ginzburg for sharing their ideas and for extensive discussions. Thanks also to Catharina Stroppel, Maurizio Martino and Olaf Schn\"urer for stimulating discussions. The author is grateful to the Max-Planck-Institut f\"ur Mathematik, Bonn for its hospitality during the writing of this paper. The author is supported by the EPSRC grant EP-H028153.

\section{Rational Cherednik algebras and the Calogero-Moser space}\label{sec:maincm}

In this section, we recall some of the basic properties of the Wilson's completion of the Calogero-Moser phase space.  

\subsection{The Calogero-Moser space}\label{sec:defncm}

The Calogero-Moser space $\CM_n$ is a completion of the phase space associated to the Calogero-Moser integrable system, which was introduced by Wilson in the seminal paper \cite{Wilson}. It is a smooth affine variety of dimension $2n$ and a symplectic manifold. Denote by $\g$ the space of all $n \times n$ matrices over $\C$ and define $\CMM_n \subset \g \times \g$ to be the set of all pairs $(X,Y)$ such that the rank of $[X,Y] + I_n$ equals one, where $I_n \in \g$ is the identity matrix. The group $\PGL_n$ acts on $\CMM_n$ by simultaneous conjugation, $g \cdot (X,Y) = (\Ad_g(X),\Ad_g(Y))$. It is shown in \cite[Corollary 1.5]{Wilson} that this action is free. 

\begin{defn}
The \tit{Calogero-Moser space} $\CM_n$ is defined to be the categorical ($= $ geometric) quotient $\CMM_n /\!/ \PGL_n$.
\end{defn}

The space $\CM_n$ is an affine symplectic manifold. 
 
\subsection{Rational Cherednik algebras of type $\mathbf{A}$}\label{sec:typeA}

In this section we recall the definition of the rational Cherednik algebra at $t = 0$ associated to the symmetric group $\s_n$ . Let $y_1, \ds, y_n$ be a basis of the $n$-dimensional space $\h$ and $x_1, \ds, x_n$ dual basis of $\h^*$. The symmetric group $\s_n$ acts on $\h$ by permuting the $y_i$'s. The rational Cherednik algebra $\H_n$ is the algebra generated by $\s_n$, $\h$ and $\h^*$, satisfying the defining relations 
$$
\sigma x_i = x_{\sigma^{-1}(i)} \sigma, \quad \sigma y_i = y_{\sigma(i)} \sigma, \quad [x_i,x_j] = [y_i,y_j] = 0, \quad \forall \ 1 \le i \neq j \le n, \ \sigma \in \s_n,
$$
\begin{equation}\label{eq:releq}
[y_i,x_j] = s_{ij}, \quad [y_i,x_i] = -\sum_{k = 1, k \neq i}^n  s_{ik}, \quad \forall \ 1 \le i \neq j \le n.
\end{equation}
The centre of $\H_n$ is denoted $\ZH_n$ and the corresponding affine variety is $\cX_n$. Let $p \in \h^*$ and denote by $\s_p$ the stabilizer of $p$ in $\s_n$. The algebra $\C[\h^*] \rtimes \s_p$ is a subalgebra of $H_n$. Each $\blambda \in \Irr (\s_p)$ can be considered a module over $\C[\h^*] \rtimes \s_p$, where $\C[\h^*]$ acts by evaluation at $p$. Then the Verma module $\Delta(p,\blambda)$ is the induced module $H_n \o_{\C[\h^*] \rtimes \s_p} \blambda$. The annihilator $I$ in $Z_n$ of $\Delta(p,\blambda)$ depends only on the image $\bb$ of $p$ in $\h^* / \s_n$. We denote by $\Omega_{\bb,\blambda}$, the closed subvariety of $X_n$ defined by $I$. It is shown in \cite{End} that $\Omega_{\bb,\blambda} \simeq \mathbb{A}^n$ is a Lagrangian subvariety of $X_n$.  

\subsection{The Etingof-Ginzburg isomorphism}

The algebra $\H_n$ is Azumaya, hence there is, up to isomorphism, a unique simple $\H_n$-module supported at each closed point of $\cX_n$. For each such $L$, denote by $\chi_L$ the corresponding character of $\ZH_n$ so that 
$$
z \cdot l = \chi_L(z) \ l, \quad \forall \ l \in L, \ z \in \ZH_n.
$$
The map $L \mapsto \chi_L$ defines a bijection between $\Irr (\H_n)$ and the closed points of $\cX_n$. Each simple module $L$ is isomorphic to the regular representation as an $\s_n$-module. Therefore, if $\s_{n-1}$ is the subgroup of $\s_n$ fixing $x_1$, then the subspace $L^{\s_{n-1}}$ is $n$-dimensional and $x_1, y_1$ act on this subspace.

\begin{prop}[\cite{EG}, Theorem 11.16]\label{prop:EGiso}
The map $L \mapsto (x_1 |_{L^{\s_{n-1}}}, y_1 |_{L^{\s_{n-1}}})$ defines an isomorphism of affine varieties $\psi_n : \cX_n \rightsim \CM_n$.
\end{prop}

\section{Factorization}

\subsection{Factorization of the Calogero-Moser space}\label{sec:CMFactor}

Let $\h$ be the subalgebra of diagonal matrices in $\g$. By Chevalley's isomorphism, we identify $\g /\!/ GL_n = \h / \s_n$. Let $\varpi : \CM_n \rightarrow \h / \s_n$ be the map that sends the pair $(X,Y)$ onto the $GL_n$-orbit of $X$. Similarly, let $\pi : \CM_n \rightarrow \h^* / \s_n$ be the map that sends $(X,Y)$ to the $\GL_n$-orbit of $Y$.

The subalgebras $\C[\h]^{\s_n}$ and $\C[\h^*]^{\s_n}$ of $H_n$ are contained in $Z_n$. The inclusions  $\C[\h]^{\s_n} \hookrightarrow Z_n$ and $\C[\h^*]^{\s_n} \hookrightarrow Z_n$ define surjective morphisms $\varpi : X_n \rightarrow \h / \s_n$ and $\pi : X_n \rightarrow \h^* / \s_n$. It follows from the proof of \cite[Theorem 10.21]{EtingofCalogeroMoser} that the following diagram commutes
\begin{equation}\label{eq:diagramt}
\xymatrix{
\cX_n \ar[rr]^{\psi_n} \ar[dr]_{\varpi \times \pi} & & \CM_n \ar[dl]^{\varpi \times \pi}\\
 & \h / \s_n \times \h^* / \s_n & 
}
\end{equation}
For $\bb \in \h^* / \s_n$, the fiber $\pi^{-1}(\bb)$ is denoted $\CM (\bb)$. Write $\bb = \sum_{i = 1}^k n_i b_i$ with $b_i \in \C$ pairwise distinct and $\sum_i n_i = n$. If $(X,Y) \in \CM (\bb)$, then we can decompose $Y = \bigoplus_{i = 1}^k Y_i$, with $Y_i$ an $n_i \times n_i$ matrix with only one eigenvalue $b_i$. We get a corresponding decomposition of $X= \bigoplus_{i = 1,j}^k X_{i,j}$. It is shown in the proof of \cite[Lemma 6.3]{Wilson} that each $(X_{i,i},Y_{i,i})$ uniquely defines a point in $\CM (n_i \cdot b_i)$. Thus, we have a map 
\begin{equation}\label{lem:W71}
\alpha_{\bb} : \CM(\bb) \ra \prod_{i = 1}^k \CM(n_i \cdot b_i).
\end{equation}
By Lemma 7.1 of \loccit, the map $\alpha_{\bb}$ is an isomorphism of affine varieties.

\subsection{Factorization of rational Cherednik algebras}

As was shown in section 5 of \cite{End}, one can use completions of the rational Cherednik algebra to prove a factorization result for the generalized Calogero-Moser space $\cX_n$. In this section we show that this factorization is compatible with isomorphism $\alpha_{\bb}$ of (\ref{lem:W71}).    

Fix $p \in \h^*$ and denote its image in $\h^* / \s_n$ by $\bb = \sum_{i = 1}^k n_i \cdot b_i$. We may assume, without loss of generality, that $p = (b_1, \ds, b_1,b_2, \ds, b_2, b_3, \ds )$. The stabilizer of $p$ with respect to $\s_n$ is $\s_p := \s_{n_1} \times \cdots \times \s_{n_k}$. The rational Cherednik algebra $\H_p$ associated to the group $\s_p$ is isomorphic to a tensor product 
\beq{eq:Hptensor}
\H_p \simeq \H_{n_1} \o \cdots \o \H_{n_k},
\eeq
and hence 
\beq{eq:HptensorZ}
Z(\H_p) \simeq \ZH_{n_1} \o \cdots \o \ZH_{n_k}.
\eeq
Therefore, Corollary 5.4 of \cite{End} implies that there is an isomorphism of affine varieties
$$
\phi : \pi^{-1}(\bb) \rightsim \prod_{i = 1}^k \pi^{-1}(n_i \cdot b_i),
$$
where $\pi^{-1}(n_i \cdot b_i)$ is a closed subvariety in $\Spec (\ZH_{n_i})$. Recall the factorization of Wilson's Calogero-Moser space as described in Lemma \ref{lem:W71}. 

\begin{thm}\label{thm:factorsagree}
The diagram
$$
\xymatrix{
\pi^{-1}(\bb) \ar[rr]^{\phi} \ar[d]_{\psi_n} & & \prod_{i = 1}^k \pi^{-1}(n_i \cdot b_i) \ar[d]^{\times_i \psi_{n_i}} \\
\CM(\bb) \ar[rr]_{\alpha_p} & & \prod_{i = 1}^k \CM(n_i \cdot b_i)
}
$$
is commutative.
\end{thm}

Before we can give the proof of Theorem \ref{thm:factorsagree}, we need to describe the isomorphism $\phi$ in representation theoretic terms. Firstly, since the diagram of the theorem involves isomorphisms between affine varieties it suffices to show commutativity on the level of closed points. The proof relies on results from section 5 of \cite{End}, and we will use freely the notation of \loccit  that we have the idempotent $e_1 \in \wh{\C[\h^*]}_{\bp}$, and if $L$ is a simple $\H_n$-module whose support is contained in $\pi^{-1}(\bb)$ then Proposition 5.3 of \loccit says that $e_1 L$ is an irreducible $\H_p$-module. Therefore, $\phi$ can be described as the map that takes the character $\chi_L$ of $\ZH_n$ to the character $\chi_{e_1 L}$ of $\ZH_{n_1} \times \cdots \times \ZH_{n_k}$, for each simple $\H_n$-module $L$ whose support is contained in $\pi^{-1}(\bb)$.

\begin{proof}
To avoid any ambiguity, the generators of $\H_p$ will be denote $\hat{x}_1, \ds, \hat{x}_n, \hat{y}_1, \ds, \hat{y}_n$, as oppose to the generators of $\H_n$, which are denoted $x_1, \ds, x_n$ and $y_1, \ds, y_n$.  

Let $N$ be a simple $\H_p$-module. Via the isomorphism (\ref{eq:Hptensor}), we write $N = N_1 \o \cdots \o N_k$, where $N_i$ is a simple $\H_{n_i}$-module. Define 
$$
m(i) = 1+ \sum_{r < i} n_r, \quad \forall \ 1 \le i \le k.
$$
Then, the isomorphism $\Spec (Z(\H_p)) \simeq \CM_{n_1} \times \cdots \times \CM_{n_k}$ that is induced from the factorization in (\ref{eq:HptensorZ}) is given on the level of closed points by the map 
$$
N \mapsto [(X_1,Y_1), \ds, (X_k,Y_k)],
$$
where $X_i$ denotes the action of $\hat{x}_{m(i)}$ on $N_{i}^{\s_{n_i - 1}}$ and $Y_i$ denotes the action of $\hat{y}_{m(i)}$ on $N_{i}^{\s_{n_i - 1}}$. Fix 
$$
W_i = \s_{n_1} \times \cdots \times \s_{n_i - 1} \times \cdots \times \s_{n_k}
$$
so that, since $N$ is the regular representation as an $\s_p$-module, one can identify $N_{i}^{\s_{n_i - 1}} = N^{W_i}$. 

Now let $L$ be a simple $\H_n$-module such that $\mf{m}_{\bb} \cdot L = 0$. Then, as explained above, the morphism $\phi$ can be described as taking $\chi_L$ to $\chi_{e_1 L}$. Therefore, to prove the commutativity of the diagram, we must show that if $(X,Y)$ represent the action of $x_1$ and $y_1$ on $L^{\s_{n-1}}$ with respect to some basis of that space then $(X_{i,i}, Y_{i,i})$ represent the action of $\hat{x}_{m(i)}, \hat{y}_{m(i)} \in \H_p$ on $(e_1 L)^{W_i}$ with respect to some basis of that space.

Let $\bb = \{ p = p_1, \ds, p_l \}$ be the orbit of $p$ under $\s_n$. Since $\mf{m}_{\bb} \cdot L = 0$, we can decompose $L$ with respect to the action of $\C[\h]$ as 
$$
L = \bigoplus_{i = 1}^l L_{p_i}.
$$
Then, the functor $e_1$ sends $L$ to $e_{1,1} L = L_{p_1}$ such that $x_i \cdot e_1 l = \hat{x}_i \cdot e_1 l$ for all $l \in L$. We can also decompose $L$ with respect to the generalized eigenspaces of the action of $y_1$:
\beq{decompLp}
L = \bigoplus_{i = 1}^k L_{b_i},
\eeq
so that $Y_i : L_{b_i}^{\s_{n-1}} \ra L_{b_i}^{\s_{n-1}}$ and $X_{i,j} : L_{b_j}^{\s_{n-1}} \ra L_{b_i}^{\s_{n-1}}$. Let us fix an $i$. Let $u_i$ denote the permutation in $\s_n$ that moves the block $[m(i), \ds , m(i+1)-1]$ to $[1,\ds, n_i]$ and moves all the entries of $[1,\ds, n]$ below $m(i)$ up by $n_i$. Then, conjugation by $u_i$ sends $W_i$ into 
$$
\widetilde{W}_i = \s_{n_i - 1} \times \s_{n_1} \times \cdots \times \s_{n_k}. 
$$
We have $\widetilde{W}_i = \s_{n-1} \cap \mathrm{Stab}_{\s_n}(u_i(p))$. 
Now 
\beq{decompLi}
L_{b_i} = \bigoplus_{p_j \in I_i} L_{p_j},
\eeq
where $I_i = \{ p_j | (p_{j})_1 = b_i \}$. We have $u_i(p) \in I_i$ and $\s_{n-1}$ acts transitively on this set. This implies that $L_{u_i(p)} \subset L_{b_i}$ such that multiplication defines an isomorphism 
\beq{eq:induip}
\Ind_{\widetilde{W}_i}^{\s_{n-1}} L_{u_i(p)} \stackrel{\sim}{\longrightarrow} L_{b_i}.
\eeq
Hence, we have an explicit isomorphism 
$$
\rho : L_{u_i(p)}^{\widetilde{W}_i} \stackrel{\sim}{\longrightarrow} L_{b_i}^{\s_{n-1}}, \quad \rho(l) = \frac{1}{(n-1)!} \sum_{\sigma \in \s_{n-1}} \sigma(l).
$$
Recall that we want to compare the action of $X_{i,i}$ and $Y_{i,i}$ on $L_{b_i}^{\s_{n-1}}$ with the action of $\hat{x}_{m(i)}, \hat{y}_{m(i)} \in \H_p$ on $(e_1 L)^{W_i} = L_{p}^{W_i}$. Since $u_i(x_{m(i)}) = x_1$ and $u_i(y_{m(i)}) = y_1$, it suffices to consider the action of $\hat{x}_1,\hat{y}_1 \in \H_{p}$ on $u_i(L_{p}^{W_i}) = L_{u(p)}^{\widetilde{W}_i}$. Thus, the theorem will follow from the following claim.  

\begin{claim}
For all $l \in L_{u_i(p)}^{\widetilde{W}_i}$, $\rho(\hat{x}_1 l) = X_{i,i} \rho(l)$ and $\rho(\hat{y}_1 l) = Y_{i,i} \rho(l)$.
\end{claim}

\begin{proof}
The action of $X_{i,i}$ and $Y_{i,i}$ on $L_{b_i}^{\s_{n-1}}$ is given by
$$
X_{i,i} : L_{b_i}^{\s_{n-1}} \stackrel{x_1 \cdot}{\longrightarrow} L \stackrel{\mathrm{pr}_i}{\longrightarrow} L_{b_i}^{\s_{n-1}}, \quad Y_{i,i} : L_{b_i}^{\s_{n-1}} \stackrel{y_1 \cdot}{\longrightarrow} L \stackrel{\mathrm{pr}_i}{\longrightarrow} L_{b_i}^{\s_{n-1}},
$$
where $\mathrm{pr}_i$ is projection onto $L_{b_i}^{\s_{n-1}}$. Since multiplication by $u_i(e_1)$ is projection onto $L_{u_i(p)}$, equation (\ref{eq:induip}) implies that $\mathrm{pr}_i$ can be expressed as multiplication by $\frac{1}{(n-1)!} \sum_{\sigma \in \s_{n-1}} \sigma(u_i(e_1))$. Therefore,
\begin{align}
\rho(\hat{x}_1 l) & = \rho(x_1 \cdot u_i(e_1) l) = \frac{1}{(n-1)!} \sum_{\sigma \in \s_{n-1}} \sigma( x_1 \cdot u_i(e_1) l) \\
 & = x_1 \frac{1}{(n-1)!} \sum_{\sigma \in \s_{n-1}} \sigma(u_i(e_1) l) \label{eq:align2}
\end{align}
A direct calculation, using the fact that the set $\{ \sigma(u_i(e_1)) \ | \ \sigma \in \s_{n-1} \}$ consists of orthogonal idempotents, shows that 
$$
\frac{1}{(n-1)!} \sum_{\sigma \in \s_{n-1}} \sigma(u_i(e_1) l) = \left( \frac{1}{(n-1)!} \sum_{\sigma \in \s_{n-1}} \sigma(u_i(e_1)) \right) \frac{1}{(n-1)!} \sum_{\sigma \in \s_{n-1}} \sigma(u_i(e_1) l). 
$$
Therefore, we have 
\begin{align*}
(\ref{eq:align2}) & = x_1 \cdot \left( \frac{1}{(n-1)!} \sum_{\sigma \in \s_{n-1}} \sigma(u_i(e_1)) \right) \frac{1}{(n-1)!} \sum_{\sigma \in \s_{n-1}} \sigma(u_i(e_1) l) \\
  & = \left( \frac{1}{(n-1)!} \sum_{\sigma \in \s_{n-1}} \sigma(u_i(e_1)) \right) x_1 \cdot \left( \frac{1}{(n-1)!} \sum_{\sigma \in \s_{n-1}} \sigma(u_i(e_1) l) \right) \\
 & = \mathrm{pr}_i (x_1 \rho(l)) = X_{i,i} \cdot \rho(l)
\end{align*}
as required. The proof of $\rho(\hat{y}_1 l) = Y_{i,i} \rho(l)$ is identical. 
\end{proof}
The statement of the theorem follows from the above claim. 
\end{proof}

\section{Torus fixed points}

If we define $\deg(x_i) = 1$, $\deg(y_i) = -1$ and $\deg (\sigma) = 0$ for $\sigma \in \s_n$, then the defining relations (\ref{eq:releq}) imply that $\H_n$ is a $\Z$-graded algebra. The centre of $\H_n$ inherits a grading. Thus, there is a natural action of $\Cs$ on $\cX_n$. Similarly, we can define an action of the torus $\Cs$-action on $\CM_n$, by setting $\alpha \cdot (X,Y) = (\alpha^{-1}, X, \alpha Y)$. The isomorphism $\psi_n$ of Proposition \ref{prop:EGiso} is $\Cs$-equivariant. Moreover, it is known that there are only finitely many fixed points in $\cX_n$ and $\CM_n$ under this action. Therefore $\psi_n$  defines a bijection between these fixed points. The purpose of this section is to calculate this bijection. The first step is to explicitly label the fixed points in $\cX_n$ and $\CM_n$ respectively. 

\subsection{$\Cs$-fixed points in $\CM_n$}\label{sec:CMfxed}

The fixed points of this $\Cs$-action were classified in \cite[\S 6]{Wilson} and explicit representatives $(X,Y;v,w)$ of each fixed point given in Lemma 6.9 of \loccit The fixed points are labeled by the partitions of $n$. For each $\lambda \vdash n$, we describe a point $\mbf{X}_\lambda \in \CM_n$. First, one rewrites $\lambda = (\lambda_1, \ds, \lambda_k)$ in Frobenius form. This means that $\lambda$ is written a union of hook partitions $(n - r + 1, 1^{r-1})$ of decreasing size such, when stack one above the other, the largest at the bottom and smallest at the top, we recover the Young diagram of $\lambda$. Combinatorially, $\lambda$ is written as an $l$-tuple of pairs $(n_1,r_1), \ds, (n_l,r_l)$ subject to the restrictions $r_i > r_j$ and $n_i - r_i > n_j - r_j$ if $i < j$. Here $\sum_i n_i = n$ and $1 \le r_i \le n_i$ for all $i$. 

Given such a pair, we have 
$$
\mbf{X}_\lambda = (X,Y) = (\oplus_{i,j} X_{i,j}, \oplus_i Y_i)_{i,j = 1 \ds l}
$$
where $Y_i$ is the upper-triangular Jordan block of size $n_i \times n_i$ with eigenvalues $0$. The matrix $X_{i,i}$ has all diagonals zero except the $-1$ diagonal (i.e. just below the main diagonal) where the entries from top left to bottom right reads
\beq{eq:list}
1,2, \ds, r_i - 1; -(n_i - r_i), \ds, -2,-1.
\eeq
For $i \neq j$, $X_{i,j}$ is a $n_i \times n_j$ matrix with non-zero entries only on the $r_j - r_i - 1$ diagonal. If $i > j$ then the non-zero diagonal of $X_{i,j}$ has $r_i$ entries equal to $n_i$ followed by $n_i - r_i$ entries equal to zero. If $i < j$, the non-zero diagonal of $X_{i,j}$ has $r_j - 1$ entries equal to $0$ followed by $n_j - r_j +1$ entries equal to $-n_i$.

\subsection{The fixed points in $\cX_n$}

Since $H_n$ is an Azumaya algebra, the closed points of $X_n$ are in bijection with isomorphism classes of simple $H_n$-modules. This implies that the $\Cs$-fixed points in $X_n$ are naturally labeled by the isomorphism classes of simple, graded $H_n$-modules. It is know \cite{Baby} that the Verma modules $\Delta(0,\lambda)$, for $\lambda$ a partition of $n$, are graded and have a unique simple graded quotient $L(\lambda)$. Moreover, up to shifts in grading, these (pairwise non-isomorphic) simple modules are all possible simple, graded $H_n$-modules. Therefore, the fixed points in $X_n$ are $\chi_{\lambda}$, where $\chi_{\lambda}$ is the character of $\ZH_n$ defined by the simple $H_n$-module $L(\lambda)$.

\begin{thm}\label{thm:fixedmatch}
The isomorphism $\psi_n : \cX_n \rightsim \CM_n$ sends the $\Cs$-fixed point $\chi_{L(\lambda)} \in \cX_n$ to the $\Cs$-fixed point $\mbf{X}_\lambda$ in $\CM_n$.
\end{thm}

\subsection{The proof of Theorem \ref{thm:fixedmatch}}\label{sec:residuemap}

The Young diagram $Y_\lambda$ of $\lambda \vdash n$ is the set 
$$
\{ (i,j) \in \Z^2 \ | \ 0 \le j \le \ell(\lambda) -1, \ 0 \le i \le \lambda_j - 1 \}.
$$
The \textit{content} of the box $(i,j)$ is $\cont(i,j) := i - j$. We define the residue of $\lambda$ to be the Laurent polynomial $\Res_{\lambda}(q) = \sum_{(i,j) \in Y(\lambda)} q^{\cont(i,j)}$. It defines a point in $\C^n / \s_n$

We denote by $\rho$ the map $\CM_n \ra \g // \GL_n \simeq \C^n / \s_n$ given by $(X,Y) \mapsto Z := YX$. It will become apparent below that this morphism is dominant. Perversely, a point $\sum_{i = 1}^k n_i \kappa_i$ in $\C^n / \s_n$ will also be thought of as a formal Laurent polynomial $\sum_{i = 1}^k n_i q^{\kappa_i}$. In particular, the polynomials $\Res_{\lambda}(q)$ define points in $\C^n / \s_n$. 

We wish to calculate the image of the fixed points $\mbf{X}_\lambda$ under $\rho$. Write $Z = YX = \oplus_{i,j} Z_{i,j}$, where $Z_{i,j}$ is a matrix of size $n_i \times n_j$. Then $Z_{i,j}$ has non-zero entries only on the $r_j - r_i$ diagonal. The square matrix $Z_{i,i}$ has entries only on the main diagonal, from top left to bottom right they are
\beq{eq:diagentries}
1,2, \ds, r_i - 1; -(n_i - r_i), \ds, -2,-1,0.
\eeq
For $i > j$, the non-zero diagonal of $Z_{i,j}$ has $r_i - 1$ entries equal to $n_i$ followed by $n_i - r_i$ entries equal to $0$. If $i < j$ then the non-zero diagonal of $Z_{i,j}$ has $r_j - 1$ entries equal to $0$ followed by $n_j - r_j + 1$ entries equal to $-n_i$. 

\begin{lem}\label{lem:rowreduce}
After row reduction $Z$ can be put in the form $\tilde{Z} = \oplus_{i,j} \tilde{Z}_{i,j}$ where $\tilde{Z}_{i,i} = Z_{i,i}$ for all $i$ and $\tilde{Z}_{i,j} = 0$ for $i > j$. 
\end{lem}

\begin{proof}
The proof is a direct calculation. If the reader really wants to understand the proof, we recommend they draw a picture to see what's going on. 

Inductively on $i$, we claim that we can remove the non-zero entries in each row of $Z_{i,j}$, where $i > j$, by taking away some multiple of a certain row above the rows of $Z_{i,j}$ in such a way that all other blocks remain unchanged. So let us fix $i > j$ and we assume by induction that $Z_{i',j'} = 0$ for all $i' < i$ and $i' > j'$. Write $Z_{i,j} = (z_{\alpha,\beta})_{\alpha, \beta}$, where $1 \le \alpha \le n_i$, $1 \le \beta \le n_j$. Then, from the above description of $Z$ we see that the only non-zero entries $z_{\alpha, \beta}$ of $Z_{i,j}$ are $z_{\alpha, \alpha + r_j - r_i}$ for $\alpha = 1, \ds , r_i - 1$ (recall that $i > j$ implies that $r_j - r_i > 0$). Now consider the column of $Z$ containing $z_{\alpha, \alpha + r_j - r_i}$. This column intersects the main diagonal of $Z$ in the block $Z_{j,j} = (\hat{z}_{a,b})_{a,b}$ and the diagonal entry of $Z_{j,j}$ in this column is $\hat{z}_{\alpha + r_j - r_i, \alpha + r_j - r_i}$. Since $\alpha \le r_i - 1$, we have $\alpha + r_j - r_i \le r_j - 1 < n_j$. Therefore, (\ref{eq:diagentries}) implies that $\hat{z}_{\alpha + r_j - r_i, \alpha + r_j - r_i} \neq 0$ and we can certainly take away from the row of $Z$ containing $z_{\alpha, \alpha + r_j - r_i}$ a multiple of the row of $Z$ containing $\hat{z}_{\alpha + r_j - r_i, \alpha + r_j - r_i}$ such that the new value of $z_{\alpha, \alpha + r_j - r_i}$ is zero. 

We claim that $\hat{z}_{\alpha + r_j - r_i, \alpha + r_j - r_i}$ is the only non-zero entry of the $(\alpha + r_j - r_i)$th row. If this is the case, then it is clear that none of the other blocks of $Z$ are changed under this row operation. The induction hypothesis implies that all entries to the left of $\hat{z}_{\alpha + r_j - r_i, \alpha + r_j - r_i}$ are zero. Since $Z_{j,j}$ is diagonal, all the entries to the right of $\hat{z}_{\alpha + r_j - r_i, \alpha + r_j - r_i}$ in $Z_{j,j}$ are also zero. Therefore, any non-zero entry of the row would lie in a block $Z_{j,k}$ with $k > j$. We have $r_k - r_j < 0$. Let $Z_{j,k} = (\overline{z}_{u,v})_{u,v}$. Then, the only non-zero entries of $Z_{j,k}$ are $\overline{z}_{u,r_k - r_j + u}$ for $u = r_j, \ds, n_j$. But the row of $Z$ containing $\hat{z}_{\alpha + r_j - r_i, \alpha + r_j - r_i}$ intersects $Z_{j,k}$ in $(\overline{z}_{\alpha + r_j - r_i,1}, \ds,\overline{z}_{\alpha + r_j - r_i,n_k})$. Now $1 \le \alpha \le r_i - 1$ so $\alpha + r_j - r_i < r_j$ which implies $\overline{z}_{\alpha + r_j - r_i,v} = 0$ for all $v$ as claimed. 
\end{proof}

\begin{prop}\label{prop:rhofixed}
The image of the $\Cs$-fixed point $\mbf{X}_{\lambda} \in \CM_n$ under $\rho$ equals $\Res_{\lambda^t}(q)$.  
\end{prop}

\begin{proof}
The argument in the proof of Lemma \ref{lem:rowreduce} still works if we replace $Z$ by $t I_n - Z$ where $t$ is some indeterminant and we work over the field $\C(t)$. Therefore, Lemma \ref{lem:rowreduce} implies that $\mathrm{det} \ ( t I_n - Z) = \prod_{a \in J} (t - a)$ where $J$ is the multiset 
$$
\bigsqcup_{i = 1}^l \ \{ 1,2, \ds, r_i - 1, -(n_i - r_i), \ds, -2,-1,0 \},
$$
when $\lambda$ in Frobenius form is $(n_1, r_1), \ds, (n_l,r_l)$. Expressed in terms of the algebra $\Z[q^{\kappa} \ | \ \kappa \in \C]$, this is $\Res_{\lambda}(q^{-1}) = \Res_{\lambda^t}(q)$.  
\end{proof}

\subsection{Degenerate affine Hecke algebras}\label{sec:degaff}

Next we construct the analogue of $\rho$ for $\cX_n$. The fact that the degenerate affine Hecke algebra is a subalgebra of the rational Cherednik algebra of type $\mbf{A}$ is well-known and has been extensively used to study the representation theory of rational Cherednik algebras at $t = 1$ e.g. \cite{ChereDiffFourier} and \cite{GriffethJack}. Martino \cite{MoDegen} has shown that this embedding of the degenerate affine Hecke algebra is also extremely useful at $t = 0$. For us, it will be used to construct a map $\rho : \cX_n \rightarrow \C^n / \s_n$. 

\bdefn
The \tit{degenerate affine Hecke algebra} $\dH_n$ is the associative algebra generated by $\C[z_1, \ds, z_n]$ and $\s_n$, satisfying the defining relations
$$
s_i z_j = z_j s_i, \quad s_i z_i = z_{i+1} s_{i} - 1,
$$
for all $i$ and $j \neq i,i+1$, where $s_i := s_{i,i+1}$.
\edefn

We note that the defining relations imply that $z_i s_i = s_i z_{i+1} - 1$. Also, as vector spaces, $\dH_n \simeq \C[z_1, \ds, z_n] \o \C \s_n$ and the centre of $\dH_n$ is the subalgebra $\C[z_1, \ds, z_n]^{\s_n}$ of symmetric functions in the $z_i$'s, see \cite{KleshBook}. The following lemma is a direct calculation.

\begin{lem}
The map 
\beq{eq:zis}
z_i \mapsto y_i x_i + \sum_{j < i} s_{i,j} = x_i y_i - \sum_{j> i } s_{i,j}, \quad \forall \ 1 \le i \le n,
\eeq
and $w \mapsto w$ for all $w \in \s_n$ defines an embedding $\dH_n \hra H_n$ such that 
$$
[z_i,x_j] = \left\{ \begin{array}{ll}
x_j s_{i,j} & j > i\\
x_i s_{i,j} & j < i\\
- \sum_{k<i} x_i s_{i,k} - \sum_{k > i} x_k s_{i,k} & i = j.
\end{array}  \right.
$$
\end{lem}

Therefore we will consider $\dH_n$ as a subalgebra of $\H_n$. Theorem 3.4 of \cite{MoDegen} says that the centre $\C[z_1, \ds, z_n]^{\s_n}$ of $\dH_n$ is contained in $\ZH_n$. The embedding $\C[z_1, \ds, z_n]^{\s_n} \hookrightarrow \ZH_n$ defines a dominant morphism $\rho : \cX_n \ra \C^n / \s_n$. A standard tool in the study of the representation theory of $\dH_n$ is induction from representations of $\C[z_1, \ds, z_n]$. Therefore, for $a \in \C^n$, define
$$
M(a) := \dH_n \o_{\C[z_1, \ds, z_n]} a,
$$
where $a$ is considered a character of $\C[z_1, \ds, z_n]$ via evaluation. The module $M(a)$ is isomorphic to the regular representation as an $\s_n$-module. Let $\mc{D}$ be the dense, open subset of $ \C^n$ consisting of all points $a = (a_1, \ds, a_n)$ such that $a_i - a_j \neq 0, \pm 1$ for all $1 \le i \neq j \le n$. Then, it is shown in \cite[Lemma 6.1.2]{KleshBook} that $M(a)$ is an irreducible $\dH_n$-module for all $a \in \mc{D}$. 

\begin{lem}\label{lem:restrictirr}
There exists a dense open subset $U$ of $\cX_n$ such that each irreducible $\H_n$-module $L$, whose support is contained in $U$, is isomorphic to $M(a)$ as a $\dH_n$-module, for some $a \in \mc{D}$. In particular, each such $L$ is irreducible as a $\dH_n$-module. 
\end{lem}

\begin{proof}
Let $U = \rho^{-1}(\overline{\mc{D}}) \subset \cX_n$, where $\overline{\mc{D}}$ denotes the image of $\mc{D}$ in $\C^n / \s_n$. Since $\overline{\mc{D}}$ is open in $\C^n / \s_n$, $U$ is open in $\cX_n$. The PBW theorems for $\H_n$ and $\dH_n$ imply that the morphism $\rho$ is dominant. Therefore, there exists a dense open subset $U'$ of $\C^n / \s_n$ such that $U' \subset \rho(\cX_n)$. Thus, $U' \cap \mc{D} \neq \emptyset$ implies that $U$ is non-empty and hence dense in $\cX_n$ because $\cX_n$ is irreducible. Let $L$ be a simple $\H_n$-module whose support is in $U$. Choose $v \in L$ to be a joint eigenvector for $z_1, \ds, z_n$. If $a_1, \ds, a_n$ are the corresponding eigenvalues of the $z_i$'s then $a = (a_1, \ds, a_n) \in \mc{D}$ and $1 \o a \mapsto v$ defines a non-zero $\dH_n$-module homomorphism $M(a) \ra L$. This is an isomorphism because $\dim M(a) = \dim L$ and $M(a)$ is irreducible. 
\end{proof}

\begin{lem}\label{lem:matrixz1}
Let $L$ be a simple $\H_n$-module such that $L \simeq M(a)$ with $a \in \C^n$ as a $\dH_n$-module. Then the eigenvalues of $z_1 \in \dH_n$ on $L^{\s_{n-1}}$ are $a_1, \ds, a_n$.
\end{lem} 

\begin{proof}
Since $L$ is isomorphic to the regular representation as a $\s_n$-module, a basis of $L^{\s_{n-1}}$ is given by $\{ e_0 s_{1,i} \o a \ | \ 1 \le i \le n \}$, where $e_0$ is the trivial idempotent in $\C \s_{n-1}$. The lemma follows from a direct calculation which shows that action of $z_1$ on $L^{\s_{n-1}}$ with respect to this basis is given by the matrix
$$
z_1 = \left( \begin{array}{cccc}
a_1 & -1 & \ldots & -1 \\
0 & a_2 & & \vdots \\
\vdots & \ddots & \ddots & -1 \\
0 & \ldots & 0 & a_n 
\end{array} \right) .
$$
Note that $s_{1,i} = s_1 \cdots s_{i-2} s_{i-1} s_{i-2} \cdots s_1 \in \s_n$. Inductively, one can show that
$$
z_1 s_1 \cdots s_{i-2} s_{i-1} = s_1 \cdots s_{i-2} s_{i-1} z_i - \sum_{j = 1}^{i-1} s_1 \cdots \wh{s}_{j} \cdots s_{i-1},
$$
where $\wh{\bullet}$ is used to denote omission. Similarly,
$$
z_i s_{i-1} s_{i-2} \cdots s_1 = s_{i-1} s_{i-2} \cdots s_1 z_1 + \sum_{j = 1}^{i-1} s_{i-1} \cdots \wh{s}_j \cdots s_1.
$$
Therefore, $z_1 s_{1,i} = s_{1,i} z_i - \sum_{j = 1}^{i-1} s_1 \cdots \wh{s}_{j} \cdots s_{i-1} s_{i-2} \cdots s_1$. Now, for $j < i-1$,
$$
\sum_{j = 1}^{i-1} s_1 \cdots \wh{s}_{j} \cdots s_{i-1} s_{i-2} \cdots s_1 = (1,i,j+1),
$$
where $(1,i,j+1)$ denotes a permutation written in cycle notation, and $s_1 \cdots \wh{s}_{i-1} s_{i-2} \cdots s_1 = 1$. Hence
$$
z_1 s_{1,i} = s_{1,i} z_i - 1 - \sum_{j = 1}^{i-2} (1,i,j+1).
$$
For each $i,j$, there exists some $k$ such that $e_0 (1, i, j+1) = e_0 s_{1,k}$. If we write 
$$
e_0 s_{1,i} = \frac{1}{(n-1)!} \sum_{\stackrel{\sigma \in \s_n}{\sigma(i) = 1}} \sigma,
$$
then clearly $k = j+1$. Thus, $z_1 e_0 s_{1,i}  = e_0 s_{1,i} z_i - e_0 - \sum_{j = 2}^{i-1} e_0 s_{1,j} = e_0 s_{1,i} z_i - \sum_{j = 1}^{i-1} e_0 s_{1,j}$. This gives the matrix form of $z_1$ described above.
\end{proof}

\begin{prop}\label{prop:equalrho}
The following diagram is commutative
$$
\xymatrix{
\cX_n \ar[rr]^{\psi_n} \ar[dr]_{\rho} & & \CM_n \ar[dl]^{\rho} \\
& \C^n / \s_n & 
}
$$
\end{prop}

\begin{proof}
Since $\psi_n$ is an isomorphism, it suffices to show that there is a dense open subset $U$ of $\cX_n$ on which the diagram is commutative. We take $U$ to be the subset of $\cX_n$ described in Lemma \ref{lem:restrictirr}. Each point in $U$ is labeled by an irreducible $\H_n$-module $L$ such that $L \simeq M(a)$ with $a \in \mc{D}$ as a $\dH_n$-module. The point $\chi_L$ labeled by $L$ is sent by $\psi_n$ to the pair $(X,Y)$, where $X = x_1 |_{L^{\s_{n-1}}}$ and $Y = y_1 |_{L^{\s_{n-1}}}$. Thus, 
$$
z_1 |_{L^{\s_{n-1}}} = y_1 x_1 |_{L^{\s_{n-1}}} = YX = Z.
$$
By definition (\ref{sec:residuemap}), $\rho \circ \psi_n (\chi_L)$ equals the eigenvalues of $Z$, which by Lemma \ref{lem:matrixz1} are $a_1, \ds, a_n$. On the other hand, $\rho(\chi_L)$ is the joint spectrum of $z_1, \ds, z_n$ on $M(a)$, which is $a_1, \ds, a_n$, because $\C[z_1, \ds, z_n]^{\s_n}$ is central in $\dH_n$.
\end{proof}

We are finally in a position to give a proof of Theorem \ref{thm:fixedmatch}. A partition is uniquely defined by its residue $\Res_{\lambda}(q)$. Therefore, Proposition \ref{prop:rhofixed} implies that $\mbf{X}_\lambda$ is uniquely defined by $\rho(\mbf{X}_\lambda)$. Hence Proposition \ref{prop:equalrho} implies that it suffices to show that $\rho(\chi_{L(\lambda)}) = \rho(\mbf{X}_\lambda)$. By Proposition \ref{prop:rhofixed}, $\rho(\mbf{X}_\lambda) = \Res_{\lambda^t}(q)$. To calculate $\rho(\chi_{\lambda})$, we need to calculate how the symmetric polynomials in the variables $z_i$ act on $L(\lambda)$. Let $w_0 \in \s_n$ be the longest word and $\Theta_i = \sum_{j < i} s_{i,j}$ the $i$th Jucys-Murphy element. Then, as noted in section 5.4 of \cite{MoDegen}, we have $- \sum_{j > i } s_{i,j} = - w_0 \Theta_i w_0$. Therefore, expression (\ref{eq:zis}) for the $z_i$ together with the arguments given in section 5.4 of \loccit imply that $\rho(\chi_{\lambda}) = \Res_{\lambda}(q^{-1})$, which equals $\Res_{\lambda^t}(q)$. 

\subsection{} Recall that the $\Cs$-fixed points in $\CM_n$ are $\mbf{X}_{\lambda}$, $\lambda \vdash n$. Define $\OmCM_\lambda := \{ \mbf{X} \in \CM_n \ |  \ \lim_{\alpha \rightarrow \infty} \mbf{X} = \mbf{X}_{\lambda} \}$. Then, it is shown in \cite[Proposition 6.11]{Wilson} that 
\beq{eq:zerofiber}
\CM(n \cdot 0) = \bigsqcup_{\lambda \vdash n} \ \OmCM_\lambda.
\eeq
If $\bb = n \cdot b_1$ for some $b_1 \in \C$ then the map $(X,Y) \mapsto (X,Y + b_1 I_n)$ defines an isomorphism $\CM(n \cdot 0) \simeq \CM(\bb)$ and (\ref{eq:zerofiber}) implies that we get a decomposition of $\CM(\bb)$ into cells $\OmCM_{b_1,\lambda}$. In general, if $\bb = \sum_{i = 1}^k n_i b_i$, then for every multipartition $\blambda = (\lambda^{(1)}, \ds, \lambda^{(k)})$ of $n$ such that $\lambda^{(i)} \vdash n_i$, define
$$
\OmCM_{\bb,\blambda} = \alpha^{-1}_{\bb} (\OmCM_{b_1,\lambda^{(1)}} \times \cdots \times \OmCM_{b_k,\lambda^{(k)}}).
$$
Then, it follows from (\ref{lem:W71}) and (\ref{eq:zerofiber}) that:

\begin{prop}\label{prop:fibersCM}
The space $\CM(\bb)$ is a finite disjoint union of affine spaces
$$
\CM(\bb) = \bigsqcup_{\blambda} \ \OmCM_{\bb,\blambda}
$$
where the union is over all multipartitions $\blambda = (\lambda^{(1)}, \ds, \lambda^{(k)})$ of $n$ such that $\lambda^{(i)} \vdash n_i$. 
\end{prop}

\section{Grassmannians}\label{sec:grasses}

Wilson constructed an embedding of the Calogero-Moser space into the adelic Grassmannian, a certain infinite dimensional (non-algebraic!) space. This embedding will allow us to identify Lagrangians $\Omega_{\bb,\blambda}$ in $X_n$ with Schubert cells in $\Grad$. Defining this embedding requires the use of several auxiliary infinite dimensional Grassmannians. In order to facilitate the reader in keeping track of all these Grassmannians, we list them here with reference to where they are first defined in the text. We have
$$
\Graad \stackrel{\sim}{\longrightarrow} \Grad \hookrightarrow \Grat \subset \tGrat,
$$
where 
\begin{itemize}
\item $\Graad$ is the Adelic Grassmannian (\ref{defn:AdelicGrass}),
\item $\Grad$ is the adelic Grassmannian (\ref{defn:iadelicGrass}), 
\item $\Grat$ is the reduced rational Grassmannian (\ref{defn:ratGrass}),\\
and 
\item $\tGrat$ is the rational Grassmannian (\ref{defn:ratGrass}). 
\end{itemize}
We also have another pair of infinite dimensional Grassmannians, the canonical Grassmannian $\QE$, defined in (\ref{defn:canonical}), which is contained inside the quasi-exponential Grassmannian $\QGr$, defined in (\ref{defn:quasiGrass}). The Grassmannians $\Graad$, $\Grad$ and $\QE$ can all be realized as an infinite union of finite dimensional spaces
$$
\Graad = \bigsqcup_{n = 1}^{\infty} \Graad_n, \quad \Grad = \bigsqcup_{n = 1}^{\infty} \Grad_n, \quad \QE = \bigsqcup_{n = 1}^{\infty} \QE_n,
$$
and we have identifications $\Graad \stackrel{\sim}{\rightarrow} \Grad \stackrel{\sim}{\leftarrow} \QE$ which restrict to 
$$
\CM_n \stackrel{\sim}{\longrightarrow} \Graad_n \stackrel{\sim}{\longrightarrow} \Grad_n \stackrel{\sim}{\longleftarrow} \QE_n.
$$
Finally, in section \ref{sec:relGrass}, we will also consider the relative Grassmannian $\Grel_n$ and comment on the embedding $\CM_n \stackrel{\sim}{\longrightarrow} \QE_n \hookrightarrow \Grel_n$. It is possible to equip most of the above spaces with topologies, making the maps between them continuous. Since this fact will not play a role in what we do, it will be easier for us simply to think of them as sets.

\subsection{The adelic Grassmannian}\label{sec:adelic}

In this section we recall the definition of the Adelic Grassmannian $\Graad$ and the adelic Grassmannian $\Grad$. Before we can do this we need to define the rational Grassmannian.

\begin{defn}\label{defn:ratGrass}
The \tit{rational Grassmannian} $\tGrat$ is the space of all $\C$-subspaces $W$ of the field $\C(z)$ such that 
\begin{enumerate}
\item there exist polynomials $a(z),b(z) \in \C[z]$ such that $a(z)^{-1} \C[z] \supseteq W \supseteq b(z) \C[z]$;
\item we have $\dim a(z)^{-1} \C[z] / W = \deg (a)$. 
\end{enumerate}
The \tit{reduced rational Grassmannian} $\Grat$ is defined to be the proper subset of $\tGrat$ consisting of those spaces $W$ such that one can chose $a(z) = b(z)$ in the above definition.
\end{defn}

The Adelic Grassmannian is defined in a similar manner: For each $b \in \C$, let $\Grass_b$ be the Grassmannian of all subspaces $W$ of $\C(z)$ such that   
\begin{enumerate}
\item there exist some $k \ge 0$ with $(z - b)^{-k} \C[z] \supseteq W \supseteq (z - b)^k \C[z]$;
\item we have $\dim (z - b)^{-k} \C[z] / W = k$. 
\end{enumerate}
The space $\C[z]$ belongs to $\Grass_b$ for all $b \in \C$. 

\begin{defn}\label{defn:AdelicGrass}
The \tit{Adelic Grassmannian} is defined to be the restricted product
$$
\Graad := \prod^0_{b \in \C} \Grass_b,
$$
where $\{ W_b \}_{b \in \C}$ belongs to $\Graad$ if and only if $W_b = \C[z]$ for all but finitely many $b \in \C$.
\end{defn}

The support of $\{ W_b \} \in \Graad$ is the finite subset of $\C$ consisting of all $b$ such that $W_b \neq \C[z]$. It is clear that each $\Grass_b$ is a subspace of $\Grat$. This can be extended to an embedding of the whole of $\Graad$ into $\Grat$. For $b \in \C \cup \{ \infty \}$, define the symmetric bilinear form $\< f,g \>_b = \mathrm{res}_{z = b} f(z) g(z) dz$ on $\C(z)$. The annihilator of a subspace $W$ of $\C(z)$ with respect to this form is written
$$
\ann_{b} W = \{ f \in \C(z) \ | \ \< f,g \>_b = 0 \ \forall \ g \in W \}.
$$
The annihilator $\ann_{\infty} W$ will be denoted $W^*$. As noted in \cite[\S 2.2]{Wilson}, the involution $W \mapsto W^*$ preserves each of the subsets $\Grass_b$ (this not true of the other $\ann_b - $). Define the embedding $i : \Graad \ra \Grat$ by 
\beq{eq:adelicembed}
i(\{ W_b \}) = \bigcap_{b \in \C} \ann_b (W^*_b) \ .
\eeq

\begin{defn}\label{defn:iadelicGrass}
The image of the $i$ inside $\Grat$ is called the \tit{adelic Grassmannian} and denoted $\Grad$. 
\end{defn}

It is shown in Lemma 5.2 of \loccit  that $i$ is indeed an embedding. One can check directly that the restriction of $i$ to $\Grass_b$ is just the naive inclusion $\Grass_p \subset \Grat$. The action of $\Cs$ on $\C(z)$ given by $\alpha \cdot z = \alpha^{-1} z$ induces an action of $\Cs$ on $\Graad$ and $\Grat$, making $i$ equivariant. 

\subsection{}\label{sec:schubertdef1}

Let $W \in \Grass_b$. Then, by definition, there exists some $N \gg 0$ such that $(z - b)^N \C[z] \subset W \subset (z - b)^{-N} \C[z]$ and $\dim W / (z - b)^N \C[z] = N$. Thus, $W / (z - b)^N \C[z]$ belongs to $\Grass(N,b)$, the Grassmannian of $N$-dimensional subspaces of $(z - b)^{-N} \C[z] / (z - b)^N \C[z]$. There is a natural stratification of $\Grass(N,b)$ into Schubert cells (to be recalled in section \ref{sec:cells}) labeled by all partitions that fit into an $N \times N$ box. Then, $W / (z - b)^N \C[z]$ will belong to a particular cell, labeled by $\lambda$ say. We define the degree of $W$ to be $|\lambda|$. One can easily check that this definition is independent of the choice of $N$. Moreover, since the degree of $\C[z] \in \Grass_b$ is $0$, the definition extends additively to the whole of $\Graad$. Let $\Graad_n$ be the set of all spaces of degree $n$ and $\Grad_n$ the image of $\Graad_n$ under $i$. There is another characterization of the space $\Grad_n$ in terms of the $\tau$-function, see section \ref{sec:tau}. Namely, $\Grad_n$ is the set of all $W$ in $\Grad$ such that $\tau_W(t_1, 0,0,\ds )$ is a polynomial of degree $n$.  

One of the key results of \cite{Wilson} is the construction of an embedding of the Calogero-Moser space into the adelic Grassmannian. Since this construction is rather technical, we will not recall the details, but simply note the features that we will require. 

\begin{thm}
There is an embedding $\beta_n : \CM_n \ra \Grad$, whose image is $\Grad_n$. 
\end{thm}

We define $\Supp : \Graad_n \rightarrow \h^* / \s_n$ by 
$$
\Supp ( \{ W_b \}) = \sum_{b \in \C} \deg(W_b) \cdot b,
$$
which, via $i$, may also be considered as a map $\Grad_n \rightarrow \h^* / \s_n$. By Theorem 7.5 of \loccit, the following diagram commutes
\beq{prop:somefactor}
\xymatrix{
\CM_n \ar[rr]^{\beta_n} \ar[dr]_{\pi} & & \Grad_n \ar[dl]^{\Supp} \\
 & \h^* / \s_n &  
}
\eeq

\subsection{Quasi-exponentials}

Recall from the introduction that $\Quasi$ denotes the space of all functions of the form $\sum_{i = 1}^k e^{b_i x} g_i(x)$, where $b_i \in \C$ and $g_i(x) \in \C[x]$. We think of the space $\Quasi$ as being a space of linear functionals on the vector space $\C[z]$ via the pairing 
\beq{eq:pairone}
\langle e^{b x} g(x), f(z) \rangle = e^{b \pa} g(\pa) \cdot f(z) |_{z = 0},
\eeq
where, formally, $e^{b \pa} \cdot z^n = (z + b)^n$. The pairing $\langle - , - \rangle$ satisfies $\< x \cdot c,f \> = \< c,\pa_z f \>$ and $\< \pa_x \cdot c,f \> = \< c, z f \>$. There is also a $\Cs$-action on $\Quasi$ given by $\alpha \cdot x = \alpha x$. The pairing $\langle - , - \rangle$ is $\Cs$-invariant. 

\begin{defn}
A finite dimensional subspace $C$ of $\Quasi$ is said to be a \tit{space of quasi-exponentials}. A quasi-exponential $f \in \Quasi$ is said to be \tit{homogeneous} if $f = e^{b x} g(x)$ for some $b \in \C$ and $g(x) \in \C[x]$. A space of quasi-exponentials $C$ is said to be \tit{homogeneous} if 
$$
C = \bigoplus_{b \in \C} C_b,
$$
where $C_b$ consists entirely of homogeneous quasi-exponentials of the form $e^{b x} g(x)$ for some $b \in \C$ and $g(x) \in \C[x]$.
\end{defn}

\begin{defn}\label{defn:quasiGrass}
The set of all homogeneous spaces of quasi-exponentials is called the \textit{quasi-exponential Grassmannian} and denoted $\QGr$. 
\end{defn}

We have $\QGr = \bigsqcup_{n = 0}^{\infty} \QGr_n$, where $\QGr_n$ is the set of all homogeneous spaces of quasi-exponentials of dimension $n$. We define $\Supp : \QGr_n \rightarrow \C^n / \s_n$ by $\Supp(C) = \sum_{i = 1}^k n_i \cdot b_i$ if $C = \bigoplus_{i = 1}^k C_{b_i}$ with $\dim C_{b_i} = n_i$. As shown in \cite[Proposition 4.6]{WilsonBi}, the spaces of quasi-exponentials are related to the rational Grassmannian as follows. For $C \subset \Quasi$, define 
$$
V_C := \{ f \in \C[z] \, | \, \langle g , f \rangle = 0, \, \forall g \in C \}.
$$ 

\begin{lem}\lab{lem:anequiv}
The subspace $W \subset \C(z)$ belongs to $\tGrat$ if and only if there exists a finite dimensional subspace $C \subset \Quasi$ and polynomial $q$ with $\deg(q) = \dim C$ such that $W = q^{-1} V_C$.
\end{lem}

\begin{proof}
Fix $C \subset \Quasi$ with $\dim C < \infty$ and $q \in \C[z]$ such that $\deg q = \dim C$. Then there exist $b_1, \ds, b_n \in \C$ and $r_1, \ds, r_n \in \N$ such that $C \subset \Span \{ e^{b_i x} x^{r_i} \}$. The polynomial  
$$
h = \prod_{i = 1}^n (z - b_i)^{r_i + 1}
$$ 
has the property that $\langle e^{b_i x} x^{r_i}, h f \rangle = 0$ for all $1 \le i \le n$ and all $f \in \C[z]$. Therefore, the ideal $h \ \C[z]$ is contained in $V_C$ and hence $hq \ \C[z] \subset V_C$ as well. Thus,
$$
h \ \C[z] = (hq)q^{-1} \C[z] \subset q^{-1} V_C \subset q^{-1} \C[z],
$$
which implies that $q^{-1} V_C \in \tGrat$. To prove the converse, we note that if $u = \prod_{i = 1}^n (z - b_i)^{a_i}$ then
$$
C_u = \{ g \in \Quasi \, | \, \langle g , uf \rangle = 0 \, \forall f \in \C[z] \} = \Span \{ e(a_i - 1,\lambda_i) \}.
$$
In particular, $\dim C_u = \deg u$. The pairing $\< - , - \>$ identifies $C_u = (\C[z] / (u))^*$. Given $p \C[z] \subset W \subset q^{-1} \C[z]$, we have $pq \C[z] \subset q W \subset \C[z]$. Set $u = pq$ and let $C = \{ \mu \, | \, \langle \mu, f \rangle = 0 \, \forall f \in qW \}$. Then, the identification $C_u \simeq (\C[z] / (u))^*$  induces an isomorphism $\Grass_{k}(C_u) \stackrel{\sim}{\rightarrow} \Grass_{\deg u - k}(\C[z] / (u))$, under which $qW = V_C$. Since $\codim_{\C[z]}(qW) = \deg q$, we have 
$$
\dim C = \dim (\overline{qW})^{\perp} = \codim_{\C[z] / (u)} (\overline{qW}) = \codim_{\C[z]}(qW) = \deg q.
$$
\end{proof}

If $C \in \QGr$ is a homogeneous space of quasi-exponentials then define 
$$
q_C(z) = \prod_{b \in \Supp(C)} (z - b)^{n_b}
$$
where $n_b = \dim C_b$. Then $\deg (q_C) = \dim C$ and, by Lemma \ref{lem:anequiv}, $q_C^{-1} V_C$ is a point in $\tGrat$. Write $\gamma : \QGr \ra \tGrat$ for the map $C \mapsto q^{-1}_C V_C$. As explained in \cite[\S 6]{WilsonBi},

\begin{prop}
The image of the map $\gamma$ equals $\Grad$.
\end{prop}

Unfortunately, as noted in \cite[\S 6]{WilsonBi}, the set of all homogeneous spaces of quasi-exponentials does not map bijectively onto $\Grad$.

\subsection{Canonical spaces}

For each $W \in \Grad$, there is a canonical choice of a space $C$ in the fiber $\gamma^{-1}(W)$. This choice allows us to define a subset of $\QGr$ such that the restriction of $\gamma$ to this subset is a bijection. 

\begin{defn}\label{defn:Wr}
Let $C \subset \Quasi$ be a homogeneous space of quasi-exponentials and fix a homogeneous basis $e^{b_1 x} g_1(x), \ds , e^{b_n x} g_n(x)$ of $C$. The \tit{Wronskian} of $C$ is defined to be 
\beq{eq:Wr}
\Wr_C (x) := \mathrm{det} \ ( \pa_x^i (e^{b_j x} g_j(x)))_{i,j = 1, \ds n} \cdot e^{- \sum_{i = 1}^n b_i x}.
\eeq
The Wronskian is (up to a scalar) independent of the choice of basis and is a polynomial in $x$. The \tit{degree} of $C$ is defined to be $\deg(C) := \deg(\Wr_C)$ and the space $C$ is said to be \tit{canonical} if $\dim C = \deg (C)$.
\end{defn}

\begin{defn}\label{defn:canonical}
The \textit{canonical Grassmannian} is defined to be the set of all canonical, homogeneous spaces of quasi-exponentials. It is denoted $\QE$. 
\end{defn}

We first show that there is a unique canonical space in $\gamma^{-1}(W)$ for all $W \in \Grass_0 \subset \Grad$. Recall from section \ref{sec:schubertdef1} that we have a partition of $\Grass_0$ into Schubert cells $\OmGr_{\lambda}$. Let $W \in \OmGr_{\lambda}$ and define $S = \{ s_0, s_1, \ds, \}$ as in (\ref{sec:schubertdef1}). We can multiply $W$ by $z^N$ for some $N \ge - s_0 = \lambda_0$ so that $z^N W \subset \C[z]$ and then take the annihilator $C$ of this space in $\Quasi$. 

\begin{lem}\label{lem:Ncan}
Let $W \in \Grass_0$ be of degree $n$ and let $r$ be the smallest positive integer such that $z^r W \subset \C[z]$. For each $N \ge r$ set $C_N = \ann_{\Quasi} z^N W$. Then, $C_n$ is the unique canonical space of quasi-exponentials in the set $\{ C_N \ | \ N \ge r \}$. 
\end{lem}

\begin{proof}
Let $\lambda$ be a partition of $n$ and assume that $W \in \OmGr_\lambda$. Then, $r = \lambda_0$. For any $N \ge \lambda_0$, the space $C_N$ is homogeneous because $z^d \C[z] \subset z^N W$ for some $d$ implies that $C_N$ consists entirely of polynomials in $x$. We claim that $\deg \Wr_{C_N}(x) = n$ for all $N \ge \lambda_0$. By definition, this claim is equivalent to the statement of the lemma. Therefore, we will give a proof of the claim. Let $r_i = s_i + N$ so that 
$$
z^{r_i} + \sum_{j = d_i + 1}^{r_{i+1} - 1} \alpha_{i,j} z^j
$$
is a basis for $z^N W$. We claim that the number of elements in $\N_0 \backslash (S + N)$ is $N$. To see this, consider the set $S + N$ as a collection of beads on $\N$. Moving all beads as far right as possible gives us the set $N + \N$. In doing so this the number of gaps does not change. Then, the claim follows from the obvious fact that $|\N \backslash (N + \N)| = N$. Write $\{ e_0 < e_1 < \ds < e_{N-1} \}$ for $\N \backslash (S + N)$. Then, $C_N$ has a basis given by 
$$
x^{e_i} + \sum_{\stackrel{j = 0}{j \neq e_k, \ \forall k}}^{e_{i} - 1} \beta_{i,j} x^j.
$$
Here, $x$ is the linear functional such that $\< x^k , f \> = \frac{1}{k!} \pa^k (f) |_{z = 0}$ so that $\< x^k , z^l \> = \delta_{k,l}$. Recall that we have chosen $d \gg 0$ such that $z^d \C[z] \subset z^N W$. The degree of the Wronskian of $C_N$ is $\sum_{i = 0}^{N-1} [e_i - (N-1-i)]$. We need to calculate this number. First, note that $\{0,1, \ds, N + d - 1 \} = \{ r_0, \ds, r_{d-1} \} \sqcup \{ e_0, \ds, e_{N-1} \}$ so that
$$
\sum_{i = 0}^{N + d - 1} i = \sum_{j = 0}^{d-1} r_i + \sum_{i = 0}^{N-1} e_i
$$
Hence
$$
\frac{(N+d)(N+d -1)}{2} -  \sum_{j = 0}^{d-1} r_i = \sum_{i = 0}^{N-1} e_i,
$$
equivalently, 
$$
\frac{(N+d)(N+d -1)}{2} -  \left( \sum_{j = 0}^{d-1} j - s_j \right) - \left( \sum_{j = 0}^{d-1} N+ j \right)  = \sum_{i = 0}^{N-1} e_i.
$$
Thus, $\frac{N(N -1)}{2} + n = \sum_{i = 0}^{N-1} e_i$ which implies that $\sum_{i = 0}^{N-1} e_i - (N- i) = n$ as claimed. In fact, one can see that making $N$ larger just makes the tail $\{ x^0, x^1 , x^2, \ds \}$ of the basis of $C_N$ longer and doesn't affect the Wronskian. 
\end{proof}

Using the fact that $\pa_x^k  e^{b x} g(x) = e^{b x} (\pa_x + b)^k g(x)$, one can check that the same argument applies to any space $W \in \Grass_b$. That is, if $W \in \OmGr_{b,\lambda}$ where $\lambda \vdash n$, then $\ann_{\Quasi} (z - b)^n W$ is the unique canonical space in the fiber $\gamma^{-1}(W)$. Therefore, we define $\eta : \Graad \ra \QE$ by 
$$
\eta( \{ W_b \}) = \bigoplus_{b \in \C} \ann_{\Quasi} \left[ (z - b)^{\deg (W_b)} W_b \right].
$$
The map $\eta$ is a bijection. 

\begin{prop}\label{prop:commute120}
The diagram
$$
\xymatrix{
\Graad \ar[rr]^\eta \ar[rd]_{i} & & \QE \ar[ld]^{\gamma} \\
 & \Grad & 
}
$$
is commutative. 
\end{prop}

\begin{proof}
Let $W = \{ W_b \} \in \Grad$ with $\eta(\{ W_b \}) = \bigoplus_{b \in \Supp (W)} C_b$. The commutativity of the diagram is the statement
$$
\left( \prod_{b \in \Supp (W)} (z - b)^{-n_b} \right) \ann_{\C[z]} \left( \bigoplus_{b \in \Supp (W)} C_b \right) = \bigcap_{b \in \C} \ann_b (W_b^*).
$$   
Since $\ann_{\C[z]} (C_b) = (z - b)^{n_b} W_b$, we must show that 
$$
\bigcap_{b \in \Supp (W)} (z - b)^{n_b} W_b = \left( \prod_{b \in \Supp (W)} (z - b)^{n_b} \right) \bigcap_{b \in \C} \ann_b (W_b^*).
$$
Since $( z - b)^{n_b} W_b \subset \C[z]$ for all $b \in \Supp (W)$, we may rewrite the above as 
\beq{eq:LHSRHS}
\bigcap_{b \in \C} (z - b)^{n_b} W_b = \left( \prod_{b \in \C} (z - b)^{n_b} \right) \bigcap_{b \in \C} \ann_b (W_b^*),
\eeq
where $n_b = 0$ for $b$ not in the support of $W$. Let LHS refer to the left hand side of equation (\ref{eq:LHSRHS}) and RHS to the right hand side of (\ref{eq:LHSRHS}). We first show that the LHS is contained in the RHS. For all $b \in \C$, we have $W_b \subseteq \ann_b(W_b^*)$. In fact, by \cite[Lemma 2.5]{Wilson}, $W_b$ is the subspace of $\ann_b(W_b^*)$ consisting of all functions whose only pole is at $b$. Let $f$ belong to the LHS. Then $f \in (z - b)^{n_b} W_b$ and hence $(z - b)^{-n_b} f \in \ann_b(W_b^*)$ for all $b$. If we take any function $g \in \ann_b(W_b^*)$ and $h \in \C(z)$ such that $h$ has no pole at $b$, then $gh \in \ann_b(W_b^*)$. This implies that $(\prod_{a \in \C} (z - a)^{-n_a}) f \in \ann_b(W_b^*)$. Hence, 
$$
\left( \prod_{a \in \C} (z - a)^{-n_a} \right) f \in \bigcap_{b \in \C} \ann_b(W_b^*) \quad \Longrightarrow \quad f \in \left( \prod_{b \in \C} (z - b)^{n_b} \right) \bigcap_{b \in \C} \ann_b (W_b^*). 
$$
Thus, LHS is contained in RHS. 

Now assume that $f \in \bigcap_{b \in \C} \ann_b (W_b^*)$. Then, $\left( \prod_{\stackrel{a \in \C}{a \neq b}} (z - a)^{n_a} \right) f$ belongs to $\ann_b (W_b^*)$ and has no poles other than at $b$. Therefore, \cite[Lemma 2.5]{Wilson} implies that $\left( \prod_{\stackrel{a \in \C}{a \neq b}} (z - a)^{n_a} \right) f$ belongs to $W_b$. Hence, 
$$
\left( \prod_{a \in \C} (z - a)^{n_a} \right) f \in (z - b)^{n_b} W_b, \ \forall \ b \in \C \quad \Longrightarrow \quad \left( \prod_{a \in \C} (z - a)^{n_a} \right) f \in \bigcap_{b \in \C} (z - b)^{n_b} W_b.
$$
Thus, RHS is contained in LHS. 
\end{proof}

Proposition \ref{prop:commute120} implies that there is a well-defined bijection $\eta \circ i^{-1} : \Grad \rightarrow \QE$. We will also denote this map by $\eta$.

\subsection{The $\tau$-function}\label{sec:tau} The rational Grassmannian is a subspace of Sato's Grassmannian and therefore plays an important role in the study of the Kadomtsev-Petviashvili (KP) hierarchy. It also means that, via the Boson-Fermion correspondence, we can associate to each $W \in \Grat$ its \tit{$\tau$-function}, which is a rational function in the infinitely many variables\footnote{Often, in the literature, one sets $x = t_1$.} $t_1,t_2, t_3, \ds$
$$
\tau_W(t_1,t_2,t_3, \ds ) \in \C(t_1,t_2,t_3, \ds ).
$$
See \cite{Miwa} for the definition of $\tau_W$. A more geometric definition of the $\tau$-function in terms of a non-vanishing section of the dual of the determinant line bundle on $\Grat$ is given in \cite{WilsonSegalLoops}. One can also define $\tau$-functions on the Calogero-Moser space $\CM_n$ and on the set of all spaces of quasi-exponentials in $\Quasi$ as follows. Let $(X,Y) \in \CM_n$ and define
\beq{eq:tausmatch}
\tau_{(X,Y)}(t_1,t_2,t_3, \ds ) = \mathrm{det} (X + \sum_{i = 1}^\infty i t_i (-Y)^{i-1} ).
\eeq
As shown in section 3.8 of \cite{Wilson}, we have $\tau_{(X,Y)} = \tau_{\beta_n(X,Y)}$.  

Let $C$ be a space of quasi-exponential and fix a basis $\{ c_1, \ds, c_n \}$ of this space. Define
$$
\tau_C^0(t_1,t_2,t_3, \ds ) = \det \left( \left\langle c_i, z^j G(z) \right\rangle \right)_{i,j = 1 \ds n},
$$
where $\< - , - \>$ is the pairing (\ref{eq:pairone}) and $G(z) := \exp \left(\sum_{i = 1}^\infty z^i t_i \right)$. Assume that $\Supp \ C = \sum_{j = 1}^k n_j b_j$ and define 
$$
\tau_C(t_1,t_2,\ds) = \left( \prod_{j = 1}^k \exp \left( - \sum_{i = 1}^{\infty} b_j^i t_i \right)^{n_j} \right) \tau_C^0 (t_1, t_2, \ds).
$$   

\begin{lem}\label{lem:cantau}
For all $C = \eta(W)$ in $\QE$, we have $\tau_W = \tau_C$ and 
\beq{eq:Wrtau}
\Wr_C(x) = \tau_C(x,0,\ds).
\eeq
\end{lem}

\begin{proof}
As shown in \cite[(5.7)]{WilsonBi}, if $\Supp C = n \cdot 0$ then $\tau_W = \tau_C^0$, which obviously is the same as $\tau_C$. The general formula will follow from \cite[Lemma 3.8]{WilsonSegalLoops}, for which we need to use the language of symmetric functions. Let $\Lambda$ be the ring of symmetric functions and denote by $p_i$, resp. $h_i,e_i$, the $i$th power, resp. complete symmetric and elementary symmetric, function in $\Lambda$. If we proclaim (see \cite[Proposition 8.2]{WilsonSegalLoops}) that  
$$
G(z)^{-1} = 1 + \sum_{i = 1}^{\infty} h_i z^i := H(z),
$$
then this forces $-it_i = p_i$, which is a consequence of the identity
$$
\exp \left( \sum_{i = 1}^{\infty} \frac{1}{i} \sum_{j = 1}^\infty t_j^i z^i \right) = \prod_{j \ge 1} \exp \left( \sum_{i = 1}^{\infty} \frac{1}{i} (t_j z)^i \right) = \prod_{j \ge 1} \frac{1}{(1 - t_j z)} = H(z).
$$
Then, $G(z)$ equals $\prod_{j \ge 1} \exp \left( - \sum_{i = 1}^{\infty} \frac{1}{i} (t_j z)^i \right) = E(-z)$, where 
$$
E(z) = 1 + \sum_{i = 1}^\infty e_i z^i = \prod_{i \ge 1} (1 + t_i z)
$$
is the generating function for the elementary symmetric functions. Set $g := G(z)$ and let 
$$
\tilde{g} = \prod_{j = 1}^k \frac{1}{(1 - b_j z^{-1})^{n_j}}.
$$
If we define $f$ and $\tilde{f}$ by $g = \exp(f)$ and $\tilde{g} = \exp(\tilde{f})$, then
$$
\tilde{f} = \sum_{i = 1}^\infty \frac{p_i(\bb)}{i} z^{-i}, \quad f = - \sum_{i = 1}^\infty \frac{p_i}{i} z^{i},
$$
where $p_i(\bb) = p_i(b_1, \ds, b_1, b_2, \ds, b_2, b_3, \ds, b_k, 0, \ds, )$ with $b_j$ occurring $n_j$ times. Then, 
$$
S(\tilde{f},f) := \frac{1}{2 \pi i} \int_{S^1} \tilde{f}'(z) f(z) dz = \sum_{i = 1}^\infty \frac{p_i(\bb)}{i} p_i.
$$
Lemma 3.8 of \cite{WilsonSegalLoops} says that, after making the substitution $-i t_i = p_i$, we have $\tau_C = \exp(S(\tilde{f},f)) \tau_C^0$. Since 
$$
\exp \left( - \sum_{i = 1}^\infty p_i(\bb) t_i \right) = \prod_{j = 1}^k \exp \left( - \sum_{i = 1}^{\infty} b_j^i t_i \right)^{n_j}, 
$$
the claim $\tau_W = \tau_C$ follows. 

Recall the definition of $\Wr_C(x)$ as given in (\ref{eq:Wr}). If one makes the substitution $t_1 = x$, $t_2 = t_3 = \ds = 0$ into $\tau_C$ then the equality (\ref{eq:Wrtau}) is evident. 
\end{proof}

\section{Schubert Cells}\label{sec:cells}

\subsection{} At various stages, we will define ``Schubert cells'' in each of the infinite Grassmannian introduced in the previous section. The notation used to denote these cells depends on which Grassmannian they sit inside, namely
$$
\OmCh_{\bullet} \subset \cX_n, \quad \OmCM_{\bullet} \subset \CM_n, \quad \OmGrad_{\bullet} \subset \Grad, \quad \OmGr_{\bullet} \subset \QGr. 
$$ 
We begin by considering the spaces $W \in \Grass_0 \subset \Grad$ i.e. those spaces $W \in \Grad$ such that $\Supp (W) = \deg(W) \cdot 0$. If $W \in \Grass_0$ then there exists an integer $k$ such that $z^{-k} \C[z] \supset W \supset z^k \C[z]$. Therefore, we can chose a basis 
\beq{eq:admisible}
\left\{ z^{s_i} + \sum_{j = s_i + 1}^{s_{i+1} - 1} \alpha_{i,j} z^j \ | \ i \in \N \right\}
\eeq
of $W$ such that $s_i = i$ for $i \gg 0$. As in \cite[\S 3]{WilsonSegalLoops}, a basis $\{ w_i \}_{i \in \N_0}$ of $W$ is said to be \textit{admissible} if $w_i = z^{i}$ for $i \gg 0$. The set (\ref{eq:admisible}) is an admissible basis. If we associate to each $w_i$, the degree $s_i$ of the trailing term of $w_i$, then we get a set $S_W = \{ s_0, s_1, \ds \}$.  The set $S$ satisfies $s_i = i$ for $i \gg 0$ and each such set corresponds to a partition $\lambda$, defined by $\lambda_i = i - s_i$ so that $\lambda_0 \ge \lambda_1 \ge \ds $ and $\lambda_i = 0$ for $i \gg 0$. The $\Cs$-fixed points in $\Grass_0$ of the action defined in section \ref{sec:adelic} are 
$$
W_\lambda = \Span \{ z^{s} \ | \ s \in S \}
$$
where $S$ is the set corresponding to $\lambda$. Then,
$$
\Grass_0 = \bigsqcup_{\lambda \in \mc{P}} \OmGrad_\lambda
$$
where $\OmGrad_\lambda = \{ W \in \Grass_0 \ | \ \lim_{\alpha \rightarrow \infty } \alpha \cdot W = W_\lambda \}$ is a Schubert cell in $\Grass_0$. It is the set of all spaces $W$ such that $S_W = \lambda$. 

For $b \in \C$, let $t_b : \C(z) \rightarrow \C(z)$ be the automorphism $z \mapsto z - b$. Then, $t_b$ defines an isomorphism $\Grass_0 \rightsim \Grass_{b}$ and we set $\OmGrad_{b,\lambda} = t_b(\OmGrad_{\lambda})$. Now let $\bb = \sum_{i = 1}^k n_i b_i \in \h^* / \s_n$ and $\blambda = (\lambda^{(1)}, \ds, \lambda^{(k)})$ a multipartition of $n$ such that $\lambda^{(i)} \vdash n_i$. We define 
$$
\OmGrad_{\bb,\blambda} = \left\{ \ i( \{ W_{b_i} \}) \ | \ \{ W_{b_i} \} \in \prod_{i = 1}^k  \OmGrad_{b_i,\lambda^{(i)}} \right\} . 
$$

\begin{lem}
For each $\bb \in \C^n / \s_n$ and $\blambda$ multipartition of type $\bb$, we have $\beta_n (\OmCM_{\bb,\blambda}) = \OmGrad_{\bb,\blambda}$.
\end{lem}

\begin{proof}
The diagram (\ref{prop:somefactor}) implies that it suffices to show that $\beta_n (\OmCM_{\lambda}) = \OmGrad_{\lambda}$. Since $\beta_n$ is $\Cs$-equivariant and both $\OmCM_{\lambda}$ and $\OmGrad_{\lambda}$ are defined to be attracting sets for the $\Cs$-action, it suffices to show that $\beta_n(\mbf{X}_{\lambda}) = W_{\lambda}$. This is shown in \cite[Proposition 6.13]{Wilson}. 
\end{proof}

\subsection{} Next we define Schubert cells in the quasi-exponential Grassmannian. We begin with the standard definition of Schubert cells in $\Grass_n(\C[x]_{2n}) \subset \QE$, where $\C[x]_{2n}$ denote the space of all polynomials in $\C[x]$ of degree less than $2n$, as given in \cite[page 147]{FultonYOungTableaux}. Let  
$$
\mc{F} = \{ 0 = \mc{F}_0 \subset \mc{F}_1 \subset \cdots \subset \mc{F}_{2n} = \C[x]_{2n} \}
$$ 
be a complete flag in $\C[x]_{2n}$. Then, given a partition $\lambda = (\lambda_0,\ds, \lambda_{n-1})$ with at most $n$ parts such that $\lambda_0 \le n$, the Schubert cell $\Omega_\lambda(\mc{F}) \subset \Grass_n(\C[x]_{2n})$ is given by 
\begin{align*}
\Omega_\lambda(\mc{F}) = \{ V \in \Grass_n(\C[x]_{2n}) \ | \ \dim (V \cap F_{k}) = i \ & \textrm{ for } n + i - \lambda_{i-1} \le k \le n + i - \lambda_{i} \\
 &  \textrm{ and all } \ 0 \le i \le n \},
\end{align*}
where the condition for $i = 0$ is $\dim (V \cap F_{n - \lambda_0}) = 0$. Then, $\dim \Omega_\lambda(\mc{F}) = n^2 - |\lambda|$. The flag at infinity is 
$$
\mc{F}(\infty) = \{ 0 \subset \C[x]_1 \subset \C[x]_2 \subset \cdots \subset \C[x]_{2n} \}.
$$
A partition $\lambda$ with at most $n$ parts such that $\lambda_0 \le n$ is precisely the same as a partition that fits into an $n \times n$ square. The compliment of $\lambda$ in this square is the rotation by $\pi$ of another partition, denoted $\overline{\lambda}$. It is the unique partition such that $\lambda_i + \overline{\lambda}_{n-i-1} = n$ for all $i = 0,1,\ds ,n-1$. For each partition $\lambda$ of $n$, we define $\OmGr_\lambda := \Omega_{\overline{\lambda}}(\mc{F}(\infty))$. It is $n$-dimensional. The $\Cs$-fixed point in $\OmGr_\lambda$ has basis $\{ x^{d_i} \ | \ i = 0, \ds, n-1 \}$, where $d_i = n + \lambda_i - (i+1)$. 

\begin{lem}\label{lem:schubertmatch}
Let $\lambda$ be a partition of $n$, then $\eta(\OmGrad_\lambda) = \OmGr_{\lambda^t}$.
\end{lem}

\begin{proof}
The proof of Lemma \ref{lem:Ncan} shows that the map $\eta : W \mapsto \ann_{\Quasi}(z^n W)$ sends the Schubert cell $\OmGrad_\lambda$ to the set $U_\lambda$ consisting of all spaces in $\Grass_n(\C[x]_{2n}) \subset \QE$ with basis 
\beq{eq:basis}
\left\{ x^{e_i} + \sum_{\stackrel{j = 0}{j \neq e_k}}^{e_{i} - 1} \beta_{i,j} x^j \ | \ 0 \le i \le n- 1 \right\}.
\eeq
Note that $\dim U_\lambda = \sum_{i = 0}^{n-1} (e_i - i) = n$. If $V \in \Grass_n(\C[x]_{2n})$ has a basis as in (\ref{eq:basis}) then $\dim (V \cap \C[x]_k) = \# \{ i \ | \ e_i < i \}$, which equals $j$ say if and only if $e_{j-1} < k \le e_{j}$. Therefore $U_\lambda = \Omega_{\overline{\mu}}(\mc{F}(\infty))$ where $\overline{\mu}$ is the partition given by $e_j = n + j - \overline{\mu}_j$. Equivalently, $e_{n-j-1} = 2n -(j+1) - \overline{\mu}_{n-j-1}$. Since $\overline{\mu}$ is defined by $\mu_j + \overline{\mu}_{n-j-1} = n$, we see that $e_{n - j - 1} = n + \mu_j - (j+1)$. Thus, from the definition of $\{ r_i \}$ and $\{ e_i \}$ given in the proof of Lemma \ref{lem:Ncan}, it follows that $\mu$ is the (unique) partition of $n$ such that
$$
\{0,1, \ds, 2n-1 \} = \{ n+ i - \lambda_i \ | \ 0 \le i \le n-1 \} \sqcup \{ n + \mu_j -(j+1) \ | \ 0 \le i \le n-1 \}.
$$
One can deduce that this implies that $\mu = \lambda^t$ from the fact that $\Z = S_\lambda \sqcup - S_{\lambda^t}$, which is easily checked.
\end{proof}

We fix coordinates on the Schubert cell $\OmGr_{\lambda}$ by fixing basis  
$$
f_i(x) = x^{e_i} + a_{i,1} z^{e_i - 1} + \ds + a_{i,0}, \quad \forall \ i = 0,1,\ds, n-1
$$
where $e_i = n + \lambda_i - (i+1)$ and $a_{i,j} \equiv 0$ if $e_{i} - j \in \{ e_{i +1}, \ds, e_{n-1} \}$, for each $C \in \OmGr_{\lambda}$. Then, $\C[\OmGr_{\lambda}]$ is a polynomial ring in the $a_{i,j}$. Let $\bb = \sum_{i = 1}^k n_i b_i \in \h^* / \s_n$ and $\blambda = (\lambda^{(1)}, \ds, \lambda^{(k)})$ a multipartition of $n$ such that $\lambda^{(i)} \vdash n_i$. Inside $\QGr$ we have the product of Grassmannians 
$$
\Grass_{\bb}(\QGr) = \Grass_{n_1} \left( e^{b_1 x} \C[x]_{2n_1} \right) \times \cdots \times \Grass_{n_k} \left( e^{b_k x} \C[x]_{2n_1} \right).
$$
As usual, we define $\OmGr_{\bb, \blambda}$ to be the product $\OmGr_{b_1,\lambda^{(1)}} \times \cdots \times \OmGr_{b_k,\lambda^{(k)}}$ in $\Grass_{\bb}(\QGr)$. The set $\Grass_{\bb}(\QGr)$ has a natural scheme structure, such that $\OmGr_{\bb, \blambda}$ is a locally closed subvariety. Moreover, 
$$
\Grass_{\bb}(\QGr) \cap \QE = \bigsqcup_{\blambda \vdash \bb} \OmGr_{\bb, \blambda}. 
$$

\subsection{} Let $W \in \Grass_0$ and fix some admissible basis $\{ w_i \}_{i \in \N_0}$ of $W$. The admissible basis may be thought of as a $\Z \times \N$ matrix, where the columns are the vectors $w_i$. Then, $W \in \Grass_n(z^{-n} \C[z] / z^n \C[z])$ if $w_i = z^{i-1}$ for all $i > n$. The corresponding matrix has the form
$$
\left( \begin{array}{ccc:ccc}
 & \vdots & & & \vdots & \\
\cdots & 0 & \cdots & \cdots & 0 & \cdots \\
 & \vdots & & & \vdots & \\
\hdashline
w_{1,-n} & \cdots & w_{n,-n} & & \vdots & \\
\vdots & \ddots & \vdots & \cdots & 0 & \cdots  \\
w_{1,n-1} & \cdots & w_{n,n-1} & & \vdots & \\
\hdashline
 & \vdots & & 1 & 0 & \\
\cdots & 0 & \cdots & 0 & \ddots & \ddots \\
 & \vdots & & & \ddots & 1 
\end{array} \right)  
$$
For each $\lambda = S \in \mc{P}$, the determinant $w^{\lambda} := w^S = \det(w_{i,j})_{i \in S, j \in \N}$ is well-defined. Also, if any $s_k \ge n$ for $k < n$ then $w^S = 0$, since the $k$th column of $(w_{i,j})_{i \in S,j \in \N}$ is the zero vector. Therefore, we may assume that $\{ s_0, \ds, s_{n-1} \}$ is a subset of the interval $[-n,n-1]$. Thus, there are ${2n \choose n}$ such $S$. Since these determinants depend, up to a scalar, on a choice of admissible basis, this means that we have defined a map $\Grass_n(z^{-n} \C[z] / z^n \C[z]) \to \mathbb{P}^{ {2n \choose n} -1}$. This is nothing but the classical Pl\"ucker embedding. In terms of partitions, each coordinate of $\mathbb{P}^{ {2n \choose n} -1}$ is labeled by a partition of length at most $n$ such that $\lambda_0 \le n$ i.e. all partitions that fit into a square of length $n$. The Pl\"ucker embedding is $\Cs$-equivariant and the fixed points $x_{\lambda}$ of the $\Cs$-action on $\Grass_n(z^{-n} \C[z] / z^n \C[z])$ are mapped to the points $w^{\lambda} = 1$ and $w^{\nu} = 0$ for all $\nu \neq \lambda$.  

If, as in the proof of Lemma \ref{lem:cantau}, we make the substitution $-i t_i = p_i$, where $p_i$ is the $i$th power polynomial in the ring $\Lambda$ of symmetric functions, then the $\tau$-function belongs to $\Lambda$. By \cite[Proposition 8.2]{WilsonSegalLoops}, the expansion of $\tau$ in terms of Schur polynomials  
$$
\tau_W = \sum_{\lambda \in \mc{P}} w^{\lambda} s_{\lambda}
$$
has coefficients given by the determinants $w^{\lambda}$. Therefore, if $W \in \Grass_{n} (z^{-n} \C[z] / z^n \C[z])$, then $\tau_W = \sum_{\lambda \in \Box} w^{\lambda} s_{\lambda}$. 

The map $\eta : \Grad \rightarrow \QE$ restricts to an isomorphism $\Grass_n(z^{-n} \C[z] / z^n \C[z]) \rightsim \Grass_n(\C[x]_{2n})$, which sends $V$ to $(z^n V)^{\perp}$. Thus, it is clearly an isomorphism of varieties. If $C \in \Grass_n(\C[x]_{2n})$, then $\tau_C = \sum_{\lambda \in \Box} c^{\lambda} s_{\lambda}$, where each $c^{\lambda}$ is a homogeneous function on $\Grass_n(\C[x]_{2n})$ which once again just defines the usual Pl\"ucker embedding. 

\begin{thm}\label{thm:grassembed}
The map $\eta \circ \beta_n : \CM_n \ra \QE \subset \QGr$ restricts to an isomorphism of algebraic varieties $\OmCM_{\bb,\blambda} \simeq \OmGr_{\bb,\blambda^t} \subset \Grass_{\bb}(\QGr)$.
\end{thm}

\begin{proof}
Since both $\eta$ and $\beta_n$ behave well with respect to factorization, by diagram (\ref{prop:somefactor}) and Proposition \ref{prop:commute120}, it suffices to show that $\eta \circ \beta_n : \CM_n \ra \QE$ restricts to an isomorphism of algebraic varieties $\OmCM_{\lambda} \simeq \OmGr_{\lambda^t} \subset \Grass_{n \cdot 0}(\C[x]_{2n})$. We expand 
$$
\tau_{(X,Y)} = \sum_{\mu \in \mc{P}} f_{\mu}(X,Y) s_{\mu},
$$
where each $f_{\lambda}(X,Y) \in \C[\CM_n]$. Define $\C[\OmGr_{\lambda^t}] \rightsim \C[\OmCM_{\lambda}]$ by $c^{\mu}(a_{i,j}) = f_{\mu}(X,Y)$. That this is well-defined and that it is an isomorphism both follow from the fact that the pair of spaces $\OmGr_{\lambda^t}$ and $\OmCM_{\lambda}$ are reduced and that the $\tau$-function distinguishes closed points of both spaces. 
\end{proof}

\subsection{The proof of Theorem \ref{thm:main}}\label{subsection:Schuberintersect}

In this subsection, we give a proof of Theorem \ref{thm:main}. We define $\nu_n : \cX_{n} \rightarrow \QGr$ to be the composition $\eta \circ \beta_n \circ \psi_n$, so that $\nu_n$ identifies $\cX_n$ with its image $\QE_n$ in $\QGr$.  Recall that Theorem \ref{thm:main} claims that $\nu_n$ restricts to an isomorphism of algebraic varieties 
$$
\nu_n : \OmCh_{\bb,\blambda} \stackrel{\sim}{\longrightarrow} \OmGr_{\bb,\blambda^t} \subset \Grass_{\bb}(\QGr). 
$$
This statement will follow from Theorem \ref{thm:grassembed}, if we can show that $\psi_n(\OmCh_{\bb,\blambda}) = \OmCM_{\bb,\blambda}$. By Theorem \ref{thm:factorsagree}, $\psi_n$ is compatible with factorizations. Therefore, it suffices to show that $\psi_n(\OmCh_{\lambda}) = \OmCM_{\lambda}$ for $\lambda$ a partition of $n$. Both $\OmCh_{\lambda}$ and $\OmCM_{\lambda}$ are attracting sets for the $\Cs$-action. Therefore, since $\psi_n$ is $\Cs$-equivariant, it suffices to show that $\psi_n(\bx_{\lambda}) = \mbf{X}_{\lambda}$. This is precisely the statement of Theorem \ref{thm:fixedmatch}, which completes the proof of Theorem \ref{thm:main}. 

\subsection{} Let $N = n^2 - 1$. The Wronskian, definition \ref{defn:Wr}, may be considered as a map $\Wr : \Grass_n(\C[x]_{2n}) \rightarrow \mathbb{P}^{N}$, where 
$$
\Wr (W) = [c_0 : \cdots : c_{N}] \quad \textrm{ if } \quad \Wr_{W}(x) = c_{N} x^{N} + \cdots + c_1 x + c_0.
$$
If $q = (q_1, \ds, q_n) \in \h$, then its image in $\h / \s_n$ is $\ba = (a_1,\ds, a_n)$ where 
$$
\prod_{i = 1}^n (x - q_i) = x^n + a_{n} x^{n-1} + \ds + a_1.
$$
We embed $\h / \s_n$ into $\mathbb{P}^{N}$, as a locally closed subvariety by 
\beq{eq:signembed}
(a_1,a_2,\ds, a_n) \mapsto [a_1 : a_2 : \cdots : a_n : 1 : 0 : \cdots : 0].
\eeq

\begin{prop}\label{prop:wronskianfiber}
The map $\nu_n : \cX_n \rightarrow \QE_n$ restricts to an isomorphism of schemes
$$
\OmCh_{0,\lambda,\ba} \simeq \Wr^{-1}(-\ba) \cap \OmGr_{\lambda},
$$
where the right hand side is the scheme theoretic interesection in $\Grass_n(\C[x]_{2n})$. 
\end{prop}

\begin{proof}
By Theorem \ref{thm:main}, $\nu_n$ restricts to an isomorphism of algebraic varieties $\OmCh_{\lambda} \simeq \OmGr_{\lambda}$. Based on diagram (\ref{eq:diagramt}), it suffices to replace $\OmCh_{0,\lambda,\ba}$ by $\OmCM_{\lambda} \cap \varpi^{-1}(\ba)$. As a locally closed embedding, $\beta_n : \OmCM_{\lambda} \hookrightarrow \Grass_{0}(\C[x]_{2n})$ was given by the polynomial coefficents of the $\tau$-function. Therefore, it suffices to show that 
\beq{eq:di123}
\xymatrix{
\OmCM_{\lambda} \ar[d]_{\varpi} \ar[r]^{\beta_n} & \OmGr_{\lambda} \ar[d]^{\Wr} \\
\h / \s_n \ar[r]_{(-1)} & \h / \s_n 
}
\eeq
commutes, as morphisms of schemes. For all $(X,Y)$ in $\CM_n$, we have $\mathrm{det} (X + \sum_{i = 1}^\infty i t_i (-Y)^{i-1} ) = \tau_{\beta_n(X,Y)}(t_1,\ds )$, see (\ref{eq:tausmatch}). Setting $t_2 = t_3 = \cdots = 0$ gives $\mathrm{det} (X + t_1) = \tau_{\beta_n(X,Y)}(t_1,0,\ds )$. Equation (\ref{eq:Wrtau}) says that $\Wr_{\beta_n(X,Y)} (t_1) = \tau_{\beta_n(X,Y)}(t_1,0,\ds )$. Thus, $\mathrm{det} (X + t_1) = \Wr_{\beta_n(X,Y)} (t_1)$. Since $\varpi(X,Y)$ is defined to be the coefficients of the polynomial $\mathrm{det} (t_1 - X)$, the diagram \ref{eq:di123} commutes.   
\end{proof}

\begin{remark}\label{rem:someiso12}
We have defined the Wronskian for any homogeneous space of quasi-exponentials. The proof of Proposition \ref{prop:wronskianfiber} shows that, as sets, we have 
$$
\nu_n(\OmCh_{\bb,\blambda,\ba}) = \Wr^{-1}(-\ba) \cap \OmGr_{\bb,\blambda} =: \OmGr_{\bb,\blambda,-\ba}
$$
for all $\ba \in \h / \s_n$, $\bb \in \h^* /\s_n$, and $\blambda$ a multipartition of type $\bb$. 
\end{remark}

\section{Baby Verma modules}\label{sec:baby}

\subsection{} Dual to the Verma modules $\Delta(p,\blambda)$ are the induced modules $\nabla(q,\bmu) = H_n \o_{\C[\h] \rtimes \s_q} \bmu$, where $q \in \h$ and $\bmu \in \Irr( \s_q)$. For each $\ba \in \h / \s_n$ and $\bb \in \h^* / \s_n$, with corresponding maximal ideals $\mf{m}_{\ba} \subset \C[\h]^{\s_n}$ and $\mf{n}_{\bb} \subset \C[\h^*]^{\s_n}$, we define the quotient modules $\Delta(p,\blambda,\ba) =  \Delta(p,\blambda) / \mf{m}_{\ba} \Delta(p,\blambda)$ and $\nabla(q,\bmu,\bb) =  \nabla(q,\bmu) / \mf{n}_{\bb} \nabla(q,\bmu)$. These are the (dual) baby Verma modules. The image of $p$ in $\h^*/W$ is denoted $\bar{p}$ and similarly for $q$.

\begin{lem}\label{lem:dimcpount}
If $\bar{p} = \bb$ and $\bar{q} = \ba$ then 
$$
\dim \Hom_{H_n}(\nabla (q,\bmu,\bb),\Delta (p,\blambda,\ba)) = \dim \Hom_{\s_n}(\Ind_{\s_q}^{\s_n} \bmu^t, \Ind_{\s_p}^{\s_n} \blambda);
$$
otherwise $\Hom_{H_n}(\nabla (q,\bmu,\bb),\Delta (p,\blambda,\ba)) = 0$.
\end{lem}

\begin{proof}
If $\bar{p} \neq \bb$, then $\mf{n}_{\bb} \cdot \Delta (p,\blambda,\ba) \neq 0$ and hence $\Hom_{H_n}(\nabla (q,\bmu,\bb),\Delta (p,\blambda,\ba)) = 0$. Therefore we assume that $\bar{p} = \bb$. Then, 
\begin{align*}
\Hom_{H_n}(\nabla (q,\bmu,\bb),\Delta (p,\blambda,\ba)) & = \Hom_{H_n}(\nabla (q,\bmu),\Delta (p,\blambda,\ba)) \\
 & = \Hom_{\C[\h] \rtimes \s_q}(\bmu, \Delta (p,\blambda,\ba)).
\end{align*}
As a $\C[\h]$-module, $\bmu$ is just the direct sum of $\dim \bmu$ copies of the skyscraper sheaf at $q$. If $a_1, \ds, a_k$ are the points in the $\s_p$-orbit corresponding to $\ba$, then 
$$
\Delta (p,\blambda,\ba) = \bigoplus_{i = 1}^k \C[\h]_{a_i} \o \Ind_{\s_p}^{\s_n} \blambda
$$
as a $\C[\h]$-module, where $\C[\h]_{a_i}$ is a module supported at $a_i$. Then, 
$$
\Hom_{\C[\h] \rtimes \s_q}(\bmu, \Delta (p,\blambda,\ba)) \subset \Hom_{\C[\h]}(\bmu, \Delta (p,\blambda,\ba))
$$
implies that $\Hom_{\C[\h] \rtimes \s_q}(\bmu, \Delta (p,\blambda,\ba)) = 0$ unless $\bar{q} = \ba$. Therefore, we assume that $\bar{q} = \ba$. Then, there exists some $i$ such that $q = a_i$ and    
$$
\Hom_{\C[\h] \rtimes \s_q}(\bmu, \Delta (p,\blambda,\ba)) = \Hom_{\C[\h] \rtimes \s_q}(\bmu, \C[\h]_{q} \o \Ind_{\s_p}^{\s_n} \blambda).
$$
Under the automorphism $x \mapsto x - q(x)$ of $\C[\h] \rtimes \s_q$, the module $\C[\h]_{q}$ is sent to the coinvariant ring $\C[\h]^{co \s_q}$ and $\bmu$ is sent to $\bmu_0$, which is defined to be the $\C[\h] \rtimes \s_q$-module isomorphic to $\bmu$ as a $\s_q$-module and supported at $0$. Thus, 
\begin{align}
\Hom_{\C[\h] \rtimes \s_q}(\bmu, \C[\h]_{q} \o \Ind_{\s_p}^{\s_n} \blambda) & \simeq \Hom_{\C[\h] \rtimes \s_q}(\bmu_0, \C[\h]^{co \s_q} \o \Ind_{\s_p}^{\s_n} \blambda)  \\
& = \Hom_{\C[\h] \rtimes \s_q}(\bmu_0, \mathrm{soc}(\C[\h]^{co \s_q}) \o \Ind_{\s_p}^{\s_n} \blambda). \label{eq:homspace}
\end{align}
The socle of $\C[\h]^{co \s_q}$ is one-dimensional, isomorphic to the sign representation $\mathrm{sgn}$ as a $\s_q$-module.  Therefore, the space (\ref{eq:homspace}) can be identified with 
$$
\Hom_{\s_q}(\bmu, \mathrm{sgn} \o \Ind_{\s_p}^{\s_n} \blambda) \simeq \Hom_{\s_n}(\Ind_{\s_q}^{\s_n} (\bmu \o \mathrm{sgn}), \Ind_{\s_p}^{\s_n} \blambda).
$$
\end{proof}

\begin{conjecture}
Let $I$ denote the annihilator of the $Z_n$-module $\Hom_{H_n}(\nabla (q,\bmu,\bb),\Verma (p,\blambda,\ba))$, and set $Z(p,q,\blambda,\bmu) = Z_n / I$. We conjecture that $Z (p,q,\blambda,\bmu)$ is a \textit{Gorenstein} ring and that the module $\Hom_{H_n}(\nabla (q,\bmu,\bb),\Verma (p,\blambda,\ba))$ is isomorphic to the coregular ($\simeq$ regular) representation as a $Z(p,q,\blambda,\bmu)$-module. 
\end{conjecture}

\begin{remark}
It is natural to expect that a suitable generalization of the above holds for rational Cherednik algebras associated to any complex refection group, provided that the intersection of $\Supp \Verma (p,\blambda,\ba)$ and $\Supp \nabla (q,\bmu,\bb)$ is contained in the smooth locus of the generalized Calogero-Moser space. 
\end{remark}

\subsection{Wilson's bispectral involution}\label{sec:bispectral}

There is a natural anti-involution $\Bi : \H_n \rightarrow \H_n^{op}$ on the rational Cherednik algebra, extending the involution $\sigma \mapsto \sigma^{-1}$ on the group algebra $\C \s_n$. It is defined by $\Bi(x_i) = y_i$, $\Bi(y_i) = x_i$ and $\Bi(s_{ij}) = s_{ij}$. This allows us to define an auto-equivalence on $\Lmod{\H_n}_{\mathrm{f.d.}}$, the category of finite dimensional $\H_n$-modules,
$$
\Bi : \Lmod{\H_n}_{\mathrm{f.d.}} \stackrel{\sim}{\longrightarrow} \Lmod{\H_n}_{\mathrm{f.d.}}, \quad \Bi(M) = M^*,
$$
where $M^*$ is the vector space dual and $(h \cdot f)(m) = f( \Bi(h) \cdot m)$ for $h \in \H_n$, $m \in M$ and $f \in M^*$. The anti-involution $\Bi$ restricts to an automorphism of $\ZH_n$ and hence of $\cX_n$. 

On the other hand, Wilson defined the \textit{bispectral involution} $b$ on $\Grad$, which in terms of Baker functions is given by $\widetilde{\psi}_W(z,x) = \widetilde{\psi}_{b(W)}(x,z)$. As noted in \cite[page 4]{Wilson}, the bispectral involution on $\CM_n$ is defined by $b(X,Y) = (Y^t,X^t)$. As one might expect, we have  

\begin{lem}\label{lem:Bipsi}
We have $\psi_n \circ \Bi = b \circ \psi_n$.   
\end{lem}

\begin{proof}
Let $L$ be a simple $\H_n$-module and $(X,Y)$ the matrices representing the action of $(x_1,y_1)$ on $L^{\s_{n-1}}$ with respect to some fixed basis. Then, with respect to the dual basis, the action of $(y_1,x_1)$ on $(L^{\s_{n-1}})^*$ is given by $(Y^t,X^t)$. 
\end{proof}

Recall (\ref{sec:degaff}) the $\Cs$-fixed points $\mbf{X}_{\lambda} \in \CM_n$. The following observation is contained in \cite{Baby}. 

\begin{lem}\label{lem:Bifix}
For all $\lambda \vdash n$, $\Bi(L(\lambda)) = L(\lambda)$. Thus, $\Bi(\mbf{X}_{\lambda}) = \mbf{X}_{\lambda}$. 
\end{lem}

\subsection{Fourier transform}

The Fourier transform, as introduced in \cite[\S 4]{EG}, is an automorphism of order four $\Fo : \H_n \stackrel{\sim}{\rightarrow} \H_n$ define by 
$$
\Fo(x_i) = y_i, \quad \Fo(y_i) = - x_i, \quad \Fo(w) = w, \quad \forall \ i \in [1,n], \ w \in \s_n. 
$$
We can use $\Fo$ to twist representations of $\H_n$. If $M$ is a $\H_n$-module then, as a vector space, ${}^{\Fo} M = M$ and the action of $\H_n$ on ${}^{\Fo} M$ is defined by $h \cdot m = \Fo(h) m$. 

\begin{lem}\label{lem:Fofix}
Choose $p \in \h^*$, $q \in \h$, $\ba \in \h/W$ and $\bb \in \h^* / W$. 
\begin{enumerate}
\item We have 
$$
{}^{\Fo} \Verma (p,\blambda) = \nabla (p,\blambda), \quad {}^{\Fo} \Verma (p,\blambda,\ba) = \nabla (p,\blambda,-\ba),
$$
$$
{}^{\Fo} \nabla (q,\bmu) = \Verma (-q,\bmu), \quad {}^{\Fo} \nabla (q,\bmu,b) = \Verma (-q,\bmu,\bb).
$$
\item Let $\lambda$ be a partition of $n$.  Then, ${}^{\Fo} L(\lambda) \simeq L(\lambda^t)$.
\end{enumerate} 
\end{lem}

\begin{proof}
Part (1) follows from the fact that $\Fo(\C[\h]) = \C[\h^*]$ and $\Fo$ acts as the identity on $\C \s_n$. 

By part (1), ${}^{\Fo} \Delta(0,\lambda,\mbf{0}) \simeq \overline{\H}_n \otimes_{\C[\h]^{co \s_n} \rtimes \s_n} \lambda$, where $\overline{\H}_n$ is the restricted rational Cherednik algebra. As a $\C[\h^*]^{co \s_n} \rtimes \s_n$-module, ${}^{\Fo} \Delta(0,\lambda,\mbf{0}) \simeq \C[\h^*]^{co \s_n} \otimes \lambda$. The socle of this module is $\det(\mathbf{y}) \otimes \lambda$, where $\det(\mathbf{y}) = \prod_{i < j} (y_i - y_j)$. Since $\det (\mathbf{y}) \otimes \lambda \simeq \lambda^t$ as an $\s_n$-module and $\h \cdot \det (\mathbf{y}) \otimes \lambda = 0$, it follows that there exists a non-zero homomorphism $\Delta(0,\lambda^t,0) \rightarrow {}^{\Fo} \Delta(0,\lambda,0)$. The image of this homomorphism is contained in the socle of ${}^{\Fo} \Delta(0,\lambda,\mbf{0})$, therefore it must factor through $L(\lambda^t)$, the simple head of $\Delta(0,\lambda^t,\mbf{0})$. The composition factors of ${}^{\Fo} \Delta(0,\lambda,0)$ are all isomorphic (since $\Delta(0,\lambda,\mbf{0})$ also has this property). Hence all these factors must be $L(\lambda^t)$. Applying $\Fo$ to the short exact sequence
$$
0 \rightarrow \Ker \rightarrow \Delta(0,\lambda,\mbf{0}) \rightarrow L(\lambda) \rightarrow 0
$$
shows that ${}^{\Fo} L(\lambda) \simeq L(\lambda^t)$.
\end{proof}

\subsection{Adjoint anti-automorphism}

Define the anti-automorphism $( - )^{\star} : \H_n \rightsim \H_n^{op}$ by $x_i^{\star} = -x_i$, $y_i^{\star} = y_i$ and $s_{i,j}^{\star} = s_{i,j}$. As in (\ref{sec:bispectral}), this defines an auto-equivalence 
$$
( - )^{\star} : \Lmod{\H_n}_{\mathrm{f.d.}} \stackrel{\sim}{\longrightarrow} \Lmod{\H_n}_{\mathrm{f.d.}}, \quad (M)^{\star} = M^*,
$$
where $M^*$ is the vector space dual and $(h \cdot f)(m) = f( h^{\star} \cdot m)$ for $h \in \H_n$, $m \in M$ and $f \in M^*$. We also have the corresponding automorphism $( -)^{\star}$ of $\cX_n$. Define the automorphism $( - )^{\star}$ of $\CM_n$ by $(X,Y) \mapsto (-X^t,Y^t)$ and recall from section \ref{sec:adelic} that $W \mapsto W^*$ defines an automorphism $( - )^*$ of $\Grad$. 

\begin{lem}\label{lem:adjoint}
We have $( - )^{\star} \circ \psi_n = \psi_n \circ ( - )^{\star}$ and $( - )^* \circ \beta_n = \beta_n \circ ( - )^{\star}$.
\end{lem}

\begin{proof}
The proof of the first statement is completely analogous to the proof of Lemma \ref{lem:Bipsi}. The second statement is \cite[Lemma 7.7]{Wilson}. 
\end{proof}

Let $\blambda = (\lambda^{(1)}, \ds, \lambda^{(k)})$ be a multipartition. The transpose of $\blambda$ is defined componentwise, $\blambda^t = ((\lambda^{(1)})^t, \ds, (\lambda^{(k)})^t)$.

\begin{prop}\label{prop:adjointiso}
Under the adjoint automorphism, $\OmCh_{\bb,\blambda}^{\star} = \OmCh_{\bb,\blambda^t}$.
\end{prop}

\begin{proof}
By Lemma \ref{lem:adjoint}, we can work either with the rational Cherednik algebra or in the Calogero-Moser space. First, we note that one can deduce from the explicit formula for $( - )^{\star}$ on $\CM_n$, together with the factorization construction given by Wilson, section \ref{sec:CMFactor}, that we have 
$$
(\OmCM_{\bb,\blambda})^{\star} = \alpha^{-1}_{\bb} ((\OmCM_{b_1,\lambda^{(1)}})^{\star} \times \cdots \times (\OmCM_{b_k,\lambda^{(k)}})^{\star}). 
$$
Moreover, for $b \in \C$ and $t_b : \CM_n \rightsim \CM_n$, $t_b(X,Y) = (X, Y - b I_n)$ we have 
$$
(\OmCM_{\lambda})^{\star} = [ t_b(\OmCM_{b,\lambda}) ]^{\star} = t_b( [\OmCM_{b,\lambda}]^{\star}).
$$
Therefore, it suffices to show that $(\OmCM_{\lambda})^{\star} = \OmCM_{\lambda^t}$. The automorphism $(- )^{\star}$ is also $\Cs$-equivariant. Hence, it suffices to show that $\mbf{X}_{\lambda}^{\star} = \mbf{X}_{\lambda^t}$. For this, we use the fact that $( - )^{\star} = \Fo \circ \Bi$. Therefore, the result follows from Lemmata \ref{lem:Bifix} and \ref{lem:Fofix}.  
\end{proof} 

\section{Intersecting Schubert cells}

\subsection{} Recall that, in addition to the Verma modules, we also defined in section \ref{sec:baby} the dual Verma modules $\nabla(q,\bmu)$. Considered as $\ZH_n$-modules, their supports were denoted $\mOCh_{\ba,\bmu}$, where $\bar{q} = \ba$ in $\h / \s_n$. In this section, we describe the sets $\nu_n(\mOCh_{\ba,\bmu})$. 

\begin{defn}\label{defn:exponents}
Let $C \in \QGr$ be an $n$-dimensional space of quasi-exponentials. Then, the \textit{sequence of exponents} of $C$ at a point $b \in \C \cup \{ \infty \}$ is the (unique) set of integers $\mbf{d} = \{ d_0 < \cdots < d_{n-1} \}$ with the property that, for each $i$, there exists a function $f \in C$ with order $d_i$ at $b$. A point $b$ of $\C \cup \{ \infty \}$ is said to be \textit{singular} if the exponents of $C$ at $b$ differs from $\{ 0, \ds, n-1\}$.
\end{defn}

Let $\ba = \sum_{i = 1}^k n_i a_i \in \h / \s_n$, where the $a_i$ are pairwise distinct. Choose a multipartition $\bmu = (\mu^{(1)}, \ds, \mu^{(k)})$ of $n$ such that $\mu^{(i)} \vdash n_i$. From $\bmu$ we define the tuple of integers $\mbf{d} = \{ d_{i,j} \ | \ i = 1,\ds, k, \ j = 0, \ds, n_i -1 \}$ by
\beq{eq:dij}
d_{i,j} := \mu^{(i)}_{n_i - j} + n_i - (j+1).
\eeq
Then, set of all $C \in \QE$ such that the singularities of $C$ are $\{ a_1, \ds, a_k \}$ and the exponents of $C$ at $a_i$ are 
$$
\{ 0 < \cdots < 2n - n_i - 1 < 2n - n_i + d_{i,1} < \cdots < 2n - n_i + d_{i,n_i} \}, 
$$
is denoted $\mOGr_{\ba,\bmu}$. The parameterization is chosen so that we can apply \cite[Theorem 2.6]{MTVDuality}, which says that
\begin{equation}\label{prop:MTVduality}
(\OmGr_{\ba,\bmu})^{\Bi} = \mOGr_{\ba,\bmu}.
\end{equation}
As a consequence, 

\begin{thm}\label{thm:bigdiagram}
For all $q \in \h$ with $\ba = \bar{q} \in \h / \s_n$, $\bmu \in W_q$ and $\bb \in \h^* / \s_n$, the map $\nu_n$ gives bijections
$$
\xymatrix{
\mOCh_{\ba,\bmu} \ar[rr]^{\sim} & & \mOGr_{\ba,\bmu} \\
\mOCh_{\ba,\bmu,\bb} \ar[rr]^{\sim} \ar@{^{(}->}[u] & & \mOGr_{\ba,\bmu,-\bb} \ar@{^{(}->}[u] 
}
$$
\end{thm}

\begin{proof}
By remark \ref{rem:someiso12}, $\nu_n(\OmCh_{\bb,\blambda,\ba}) = \OmGr_{\bb,\blambda^t,-\ba}$. As noted in section \ref{sec:bispectral}, the map $\nu_n$ intertwines the bispectral involution on $\cX_n$ with Wilson's bispectral involution on $\QE$ (or rather the corresponding integral transform as defined in \cite{MTVDuality}). Equation (\ref{prop:MTVduality}) implies that $(\OmGr_{\ba,\bmu,\bb})^{\Bi} = \mOGr_{\ba,\bmu,\bb}$. 

Let $(-1) : \H_n \rightsim \H_n$ be the isomorphism which is the identity on $\s_n$ and maps $x_i$ to $-x_i$ and $y_j$ to $-y_j$. Then, $\Bi = \Fo \circ ( - )^{\star} \circ (-1)$. We have $\OmCh_{\ba,\bmu,\bb}^{(-1)} = \OmCh_{-\ba,\bmu,-\bb}$. Proposition \ref{prop:adjointiso} implies that $\OmCh_{-\ba,\bmu,-\bb}^{\star} = \OmCh_{-\ba,\bmu^t,\bb}$ and Lemma \ref{lem:Fofix} (1) implies that $\OmCh_{-\ba,\bmu^t,\bb}^{\Fo} = \mOCh_{\ba,\bmu^t,\bb}$. Therefore, $\OmCh_{\ba,\bmu,\bb}^{\Bi} = \mOCh_{\ba,\bmu^t,\bb}$. This implies the claim of the theorem.  
\end{proof}

\subsection{} For $q_i \in \C$, let $\mc{F}(q_i)$ be the complete flag 
$$
\mc{F}_{\bullet}(q_i) \ : \ \mc{F}_j(q_i) = (x - q_i)^{2n-j} \C[x]_j, \quad 0 \le j \le 2n,
$$
in $\C[x]_{2n}$. Let $q = (q_1, \ds, q_1, q_2, \ds, q_2, q_3, \ds )$, where the $q_i$ are pairwise distinct and $q_i$ occurs $n_i$ terms. Let $\bmu = (\mu^{(1)}, \ds, \mu^{(k)})$ be a multipartition with $\mu^{(i)} \vdash n_i$; equivalently $\bmu \in \Irr (\s_q)$. Then, we define 
$$
\Omega_{\bmu}(q) = \bigcap_{i = 1}^k \Omega_{\mu^{(i)}}(\mc{F}(q_i)),
$$
a scheme theoretic intersection of Schubert cells in $\Grass_n(\C[x]_{2n})$. Let $\Grass_n(\C[x]_{2n})_{\mathrm{can}}$ denote the intersection of $\Grass_n(\C[x]_{2n})$ with $\QE$ in $\QGr$, considered as a reduced variety. Then, $\Grass_n(\C[x]_{2n})_{\mathrm{can}} = \bigsqcup_{\lambda \vdash n} \OmGr_{\lambda}$. We define $\Omega_{\bmu}(q)_{\mathrm{can}}$ to be the scheme-theoretic intersection of $\Omega_{\bmu}(q)$ with $\Grass_n(\C[x]_{2n})_{\mathrm{can}}$. In order to prove a special case of the above conjecture when $W = \s_n$, we need to make the following technical assumption. 

\begin{assumption}\label{assume:schubert}
We have an equality $\nu_n(\mOCh_{\ba,\bmu,n \cdot 0}) = \Omega_{\bmu}(q)_{\mathrm{can}}$ as \textit{subschemes} of $\Grass_n(\C[x]_{2n})$. 
\end{assumption}

We remark that neither $\mOCh_{\ba,\bmu,n \cdot 0}$ or $\Omega_{\bmu}(q)_{\mathrm{can}}$ is a reduced scheme. In order to convince the reader that assumption \ref{assume:schubert} is not unreasonable, we have 

\begin{lem}\label{lem:dimcount1}
Let $q \in \h$ and $\bmu \in \Irr (\s_q)$. Then, we have an equality $\nu_n(\mOCh_{\ba,\bmu,n \cdot 0}) = \Omega_{\bmu}(q)_{\mathrm{can}}$ of subsets of $\Grass_n(\C[x]_{2n})$ and 
$$
\dim \C[\mOCh_{\ba,\bmu,n \cdot 0}] = \dim \C[\Omega_{\bmu}(q)_{\mathrm{can}}]  = |\s_n / \s_q | \dim \bmu.
$$
\end{lem}

\begin{proof}
A point $V \in \Grass_n(\C[x]_{2n})$ belongs to $\Omega_{\mu^{(i)}}(q_i)$ if and only if $q_i$ is a singular point of $V$ such that the exponents of $V$ at $q_i$ are encoded by $\mu^{(i)}$. On the other hand, Theorem \ref{thm:bigdiagram} implies that $\nu_n(\mOCh_{\ba,\bmu,n \cdot 0})$ is the set of all \textit{canonical} homogeneous spaces of quasi-exponentials with exponents prescribed by $q$ and $\bmu$ contained in $\Grass_n(\C[x]_{2n})$. Every space in $\Grass_n(\C[x]_{2n})$ is obvious homogeneous. Therefore $\nu_n(\mOCh_{\ba,\bmu,n \cdot 0})$ is the intersection of $\Omega_{\bmu}(q)$ with $\Grass_n(\C[x]_{2n})_{\mathrm{can}}$, which by definition is $\Omega_{\bmu}(q)_{\mathrm{can}}$. 

Theorem 1.2 of \cite{End} implies that $\dim \C[\mOCh_{\ba,\bmu,n \cdot 0}]$ equals the rank of $e \nabla (q,\bmu)$ as a free $\C[\h^*]^{\s_n}$-module. As a $\C[\h^*]^{\s_n}$-module, 
$$
e \nabla (q,\bmu) \simeq e ( \C[\h^*] \o \Ind_{\s_q}^{\s_n} \bmu) \simeq e (\Ind_{\s_q}^{\s_n} (\C[\h^*] \o \bmu)) \simeq e_q (\C[\h^*] \o \bmu),
$$
where $e_q$ is the trivial idempotent in $\C \s_{q}$. This implies that the rank of $e \nabla (q,\bmu)$ equals $|\s_n / \s_q | \dim \bmu$.  

Recall that two complete flags $\mc{F}_{\bullet}$ and $\mc{G}_{\bullet}$ in $\C[x]_{2n}$ are said to be transverse if $\dim \mc{F}_i \cap \mc{G}_j  = \min \{ i + j - 2n, 0 \}$ for all $i,j$. Let $b \neq c \in \C \cup \{ \infty \}$. Then, it is easy to check that the flags $\mc{F}_{\bullet}(b)$ and $\mc{F}_{\bullet}(c)$ are transverse. Hence, the flags appearing in the intersection $\Omega_{\bmu}(q)$ are pairwise transverse. They are also transverse to $\Omega_{\overline{\lambda}}(\mc{F}(\infty))$ for each partition $\lambda$ of $n$. As noted above, $\Grass_n(\C[x]_{2n})_{\mathrm{can}} = \bigsqcup_{\lambda \vdash n} \Omega_{\overline{\lambda}}(\mc{F}(\infty))$. Since the set-theoretic intersection $\Omega_{\bmu}(q) \cap \Omega_{\overline{\lambda}}(\mc{F}(\infty))$ consists of finitely many points (notice that $\dim \Omega_{\mu^{(i)}}(\mc{F}(q_i)) = n^2 - |\mu^{(i)}|$ and $\dim \Omega_{\overline{\lambda}}(\mc{F}(\infty)) = |\lambda|$, hence if $\sum_{i} \dim \mu^{(i)} = n$ and $\lambda \vdash n$, then $\dim \Omega_{\bmu}(q) \cap \Omega_{\overline{\lambda}}(\mc{F}(\infty)) = 0$) the transeverality condition implies that 
$$
[\Omega_{\mu^{(1)}}(\mc{F}(q_1))] \cdots [\Omega_{\mu^{(k)}}(\mc{F}(q_k))] \cdot [\Omega_{\overline{\lambda}}(\mc{F}(\infty))]
$$
is some multiple of the identity in the cohomology ring $H^{*}(\Grass_0(\C[x]_{2n}))$, where $[X] \cdot [Y]$ denotes multiplication in $H^{*}(\Grass_0(\C[x]_{2n}))$ of the classes defined by the \textit{closures} of the locally closed subvarieties $X,Y$ of $\Grass_0(\C[x]_{2n})$. Thus, 
$$
\dim \C[\Omega_{\bmu}(q)_{\mathrm{can}}] = \sum_{\lambda \vdash n}  [\Omega_{\mu^{(1)}}(\mc{F}(q_1))] \cdots [\Omega_{\mu^{(k)}}(\mc{F}(q_k))] \cdot [\Omega_{\overline{\lambda}}(\mc{F}(\infty))].  
$$
Let $\sigma_{\lambda} = [\Omega_{\lambda}(\mc{F}(b))] = [\Omega_{\lambda}(\mc{F}(\infty))]$ be the class of a Schubert cell in $H^{*}(\Grass_0(\C[x]_{2n}))$. They form a basis of $H^{*}(\Grass_0(\C[x]_{2n}))$ such that $\sigma_{(n,\ds,n)} = 1$. Let $\langle - , - \rangle$ be the non-degenerate pairing on $H^{*}(\Grass_0(\C[x]_{2n}))$ defined by letting $\langle [X] , [Y] \rangle$ be the coefficient of $1$ in the expansion of $[X] \cdot [Y]$ in terms of the basis $\{ \sigma_{\lambda} \}$. The duality theorem, \cite[page 149]{FultonYOungTableaux}, says that $\langle \sigma_{\lambda} , \sigma_{\overline{\mu}} \rangle = \delta_{\lambda,\mu}$. Thus, Schubert calculus implies that 
\begin{align*}
\dim \C[\Omega_{\bmu}(q)_{\mathrm{can}}] & = \sum_{\lambda \vdash n}  \langle \sigma_{\mu^{(1)}} \cdots \sigma_{\mu^{(k)}} , \sigma_{\lambda} \rangle \\
 & = \sum_{\lambda \in \Irr (\s_n)} \Hom_{\C \s_n}(\lambda, \Ind_{\s_q}^{\s_n} \bmu) = \dim  \Ind_{\s_q}^{\s_n} \bmu,
\end{align*}
as required. 
\end{proof}

\begin{thm}\label{thm:intersectassume}
Under assumption \ref{assume:schubert}, the map $\nu_n$ induces an isomorphism of Gorenstein rings 
$$
Z(0,q,\blambda,\bmu) \simeq \C[\OmGr_{0,\lambda,-\ba} \cap \Omega_{\bmu}(q)],
$$
and $\Hom_{\H_{n}}(\nabla (q,\bmu,\mbf{0}),\Verma (0,\blambda,\ba))$ is the coregular representation as a $\ZH (0,q,\lambda,\bmu)$-module. 
\end{thm}

\begin{proof}
The Morita equivalence between $\H_n$ and $\ZH_n$ implies that 
$$
\Hom_{\H_{n}}(\nabla (q,\bmu,0),\Verma (0,\lambda,\ba)) \simeq \Hom_{\ZH_n}(e \nabla (q,\bmu,0), e \Verma (0,\lambda,\ba)).
$$
Let $I$ be the annihilator of $e \nabla (q,\bmu,0)$ in $\ZH_n$ and $J$ the annihilator of $e \Verma (0,\lambda,\ba)$. Theorem 1.2 of \cite{End} implies that $e \nabla (q,\bmu,0) \simeq \ZH_n / I$ and $e \Verma (0,\lambda,\ba) \simeq \ZH_n / J$ are cyclic $\ZH_n$-modules. By Proposition \ref{prop:wronskianfiber}, $\Spec \ZH_n / J = \OmCh_{0,\lambda,\ba}$ is isomorphic to $\Wr^{-1}(-\ba) \cap \OmGr_{\lambda} = \OmGr_{0,\lambda,-\ba}$. Using assumption \ref{assume:schubert}, we have $\Spec \ZH_n / I = \mOCh_{\ba,\bmu,n \cdot 0} \simeq \Omega_{\bmu}(q)_{\mathrm{can}}$. Therefore, \cite[Lemma 4.3]{SchubertGLN} implies that
$$
\ZH_n / (I + J) \simeq \C[\Wr^{-1}(-\ba) \cap \OmGr_{\lambda} \cap \Omega_{\bmu}(q)_{\mathrm{can}}] = \C[\Wr^{-1}(-\ba) \cap \OmGr_{\lambda} \cap \Omega_{\bmu}(q)]
$$
is a Gorenstein ring. This proves the first statement of the theorem. The result \cite[Lemma 3.8]{SchubertGLN} states that:   

\begin{claim}
Let $Z$ be a commutative ring and $I,J$ ideals of $Z$ such that
\begin{itemize}
\item $\dim Z/I, Z/J < \infty$,
\item $Z/J$ and $Z/ I + J$ are Gorenstein. 
\end{itemize}
Let $\overline{I} = I + J$ in $Z/J$. Then, $\ker \overline{I} \simeq (Z / (I + J))^*$ as $Z / (I + J)$-modules. 
\end{claim}

Applying the above claim in our case, it suffices to identify $\ker \overline{I}$ with 
$$
\Hom_{\H_{n}}(\nabla (q,\bmu,0),\Verma (0,\lambda,\ba)) = \Hom_{\ZH_n}(\ZH_n / J , \ZH_n / I). 
$$
This is straight-forward. 
\end{proof}

\begin{remark}
The claim about dimensions made after Corollary \ref{cor:intersect} can be deduced from the proof of Lemmata \ref{lem:dimcpount} and \ref{lem:dimcount1}. Also, $\Hom_{\H_{n}}(\nabla (q,\bmu,0),\Verma (0,\lambda,\ba)) = \Hom_{\H_{n}}(\nabla (q,\bmu),\Verma (0,\lambda,))$, which implies that Corollary \ref{cor:intersect} is equivalent to the statement of Theorem \ref{thm:intersectassume}. 
\end{remark}

\section{The relative Grassmanian}\label{sec:relGrass}

\subsection{} In this final section we make some basic remarks about the relative Grassmanian, and its relationship to the Calogero-Moser space. As noted in \cite{CMFix}, one can interpreter Wilson's embedding $\beta_n$ as an embedding of $\CM_n$ into 
$$
\Grel_n := \{ (I,W) \ | \ I \lhd \C[z] \textrm{ with $\dim \C[z] / I = n$ and $W \subset \C[z] / I^2$ an $n$-dimensional subspace.} \},
$$
the \textit{relative Grassmaniann}. Since both $\CM_n$ and $\Grel_n$ are quasi-projective varieties, it is natural to expect that Wilson's embedding is a morphism of varieties. In this subsection we suggest one way that one might hope to show this. Projection onto $I$ defines a proper map $\Grel_n \rightarrow \mathbb{A}^{(n)} = \mathrm{Hilb}^n (\C)$. Let $E$ be the rank $2n$ vector bundle on $\mathbb{A}^{(n)}$, whose fiber over $I$ is $\C[z] / I^2$. Recall, \cite[Example 2.2.3]{HuybrechtsLehn}, that the relative Grassmanian is the space that represents the contravariant functor $F : \mathrm{Sch}_{\mathbb{A}^{(n)}} \rightarrow \mathrm{Sets}$, from the category of schemes over $\mathbb{A}^{(n)}$ to sets defined by 
$$
F(X) = \{ \phi : \xi^* E \twoheadrightarrow \mc{F} \ | \ \mc{F} \textrm{ flat of rank $n$ } \} / \simeq.
$$
where $\xi : X \rightarrow \mathbb{A}^{(n)}$.  

We denote by $R$ the coordinate ring of $\CM_n$. Recall that $\pi : \CM_n \rightarrow \h^* / \s_n$. Let $\mc{E} = \pi^* E$ be the vector bundle of rank $2n$ on $\CM_n$ induced by $E$. Since $\CM_n$ is affine, we consider instead the corresponding projective $R$-module of section, which we will also denote by $\mc{E}$. Since $\mc{E}$ is the pull-back of a projective $\C[\mathbb{A}^{(n)}]$-module, it is actually a free $R$-module. Explicitly,
$$
\mc{E} = R[z] / (\mathrm{det} ( z- Y)^2).
$$

Associated to each space $W \in \Grad$ is the Baker function $\widetilde{\psi}_W(z,x)$, see \cite{WilsonBi} and \cite{Wilson}. Just as for the $\tau$-function, the Baker function distinguishes points in that $\widetilde{\psi}_{W_1}(z,x) = \widetilde{\psi}_{W_2}(z,x)$ if and only if $W_1 = W_2$. If $\Supp( W) = \sum_{i = 1}^k n_i b_i$, then define $\Psi_W(z) = \prod_{i = 1}^k (z - b_i)^{n_i}$. The \textit{regular} Baker function $\psi_W(z,x)$ is defined to be $\Psi_W(z) \widetilde{\psi}_W(z,x)$. We define the \textit{polynomial} Baker function to be 
$$
\psi^{\mathsf{pol}}_W(z,x) = \Wr_W(x) \cdot \psi_W(z,x) = \Psi_W(z) \cdot \Wr_W(x) \cdot \widetilde{\psi}_W(z,x).
$$
It is known, e.g. as a consequence of \cite[Proposition 6.5]{WilsonBi}, that $\psi_W^{\mathsf{pol}}(z,x) = g(z,x) e^{zx}$, where $g(z,x)$ is a polynomial of degree $\deg (W)$ in both $z$ and $x$. The following lemma follows from the description of $\widetilde{\psi}_W$ given in section 4 of \cite{WilsonBi}.

\begin{lem}\label{lem:Wdiff}
Let $W \in \Grad$ and $C = \eta(W) \in \QE$. Then, 
$$
\Psi_W(z) W = \Span \{ (\pa_x^k \psi_W^{\mathsf{pol}}(z,x)) |_{x = 0} \ \textrm{ for all $k \in \N$} \} = C^{\perp}.
$$
\end{lem}

For each $(X,Y) \in \CM_n$, consider the element $\psi_{(X,Y)}^{\mathsf{pol}} = e^{xz} \det((X - x)(Y - z) - 1) \in R \hat{\o} \C[[x,z]]$. Let $\mc{K}$ be the $R$-submodule of $\mc{E}$ generated by 
$$
\pa_x^0 \psi_{| x = 0}^{\mathsf{pol}}, \pa_x^1 \psi_{| x = 0}^{\mathsf{pol}}, \pa_x^2 \psi_{| x = 0}^{\mathsf{pol}}, \ds
$$
Then, we define $\mc{F}$ to be the quotient $\mc{E} /\mc{K}$. 

\begin{conj}
The quotient $\mc{E} \twoheadrightarrow \mc{F}$ is a vector bundle of rank $n$ on $\CM_n$, inducing a locally closed embedding $\beta_n : \CM_n \rightarrow \Grel_n$. 
\end{conj}

\begin{remark}
The definition of $\CM_n$ as a G.I.T. quotient implies that there is a ``tautological'' rank $n$ bundle on the space. It is unclear to the author how this tautological bundle is related to $\mc{F}$. 
\end{remark}

Expanding, $\psi_W^{\mathsf{pol}}(z,x)e^{-zx} = \sum_{i, j = 0}^n a_{i,j} z^i x^j$, we write $D_W = \sum_{i,j = 0}^n a_{i,j} x^j \pa_x^i$.

\begin{lem}
Let $W \in \Grad$. Then, $C = \eta(W) \in \QE$ is the space of all holomorphic solutions of the differential equation $D_W$. 
\end{lem}

\begin{proof}
By Lemma \ref{lem:Wdiff}, $\Psi_W(z) W = C^{\perp}$, which equals $\Span \{ (\pa_x^k \psi_W^{\mathsf{pol}}(z,x)) |_{x = 0} \ | \ k = 0,1,\ds \}$. We apply the easy identity $(\pa_x^k x^j e^{xz}) |_{x = 0} = \pa_z^j(z^k)$. Thus, $c \in C$ if and only if 
\begin{align*}
\langle c, (\pa_x^k \psi_W(z,x)) |_{x = 0} \rangle  & = \left\langle c, \sum_{i,j = 0}^n a_{i,j} z^i (\pa_x^k (x^j) e^{zx}) |_{x = 0} \right\rangle \\
 & = \left\langle c, \sum_{i,j = 0}^n a_{i,j} z^i \pa_z^j(z^k) \right\rangle = \left\langle \sum_{i,j = 0}^n a_{i,j} x^j \pa_x^i c, z^k \right\rangle = 0
\end{align*}
for all $k \in \mathbb{N}$. This implies that $\sum_{i,j = 0}^n a_{i,j} x^j \pa_x^i c = 0$. Since the dimension of $C$ is $n$, $C$ contains all solutions of the differential equation $D_W$. 
\end{proof}

If $g(x)$ is a polynomial and $p \neq 0$, then the function $e^{p x} g(x)$ has an irregular singularity of order one at infinity. Thus, if $D$ is an $n$th order differential equation whose space of solutions is $C \in \QGr$ then $D$ has only regular singularities in $\C$ and (at worst) an irregular singularity at $\infty$ of order one. Moreover, the residue of $D$ at $\infty$ is $\Supp (C) \in \h^* / \s_n$. Recall that $D$ is said to be \textit{Fuchsian} if it has only regular singularities i.e. if and only if $\Supp (C) = 0$.  Given a simple $H_n$-module $L$, we write $D_L$ for the $n$th order differential equation whose space of solutions equals $\nu_n(\chi_L) \in \QGr$. 

\begin{cor}
Let $L$ be a simple $H_n$-module. Then, the differential equation $D_L$ is Fuchsian if and only if $\C[\h^*]_+^{\s_n} \cdot L = 0$. 
\end{cor}

\begin{proof}
The space $C$ of solutions of $D_L$ is a homogeneous space of quasi-exponential functions. As noted above, $D_L$ will be Fuchsian if and only if the support of $C$ equals zero. That is, if and only if the augmentation ideal in $\C[\h^*]^{\s_n}_+$ annihilates $L$. 
\end{proof}

\begin{example}
For each partition $\lambda$ of $n$ we have the simple $\H_n$-module $L(\lambda)$. Since the support of $L(\lambda)$ is sent to the $\Cs$-fixed point in $\OmGr_{\lambda}$, the proof of Lemma \ref{lem:schubertmatch} shows that  
$$
D_{L(\lambda)} = \prod_{i = 0}^{n-1} ( x \pa - e_i),
$$
where $e_i = n + \lambda_i- (i+1)$. 
\end{example}


\begin{thebibliography}{10}

\bibitem{End}
G.~Bellamy.
\newblock {E}ndomorphisms of {V}erma modules for rational {C}herednik algebras.
\newblock {\em In preperation}, 2013.

\bibitem{ChereDiffFourier}
I.~Cherednik.
\newblock Double affine {H}ecke algebras and difference {F}ourier transforms.
\newblock {\em Invent. Math.}, 152(2):213--303, 2003.

\bibitem{EtingofCalogeroMoser}
P.~Etingof.
\newblock {\em Calogero-{M}oser {S}ystems and {R}epresentation {T}heory}.
\newblock Zurich Lectures in Advanced Mathematics. European Mathematical
  Society (EMS), Z\"urich, 2007.

\bibitem{EG}
P.~Etingof and V.~Ginzburg.
\newblock Symplectic reflection algebras, {C}alogero-{M}oser space, and
  deformed {H}arish-{C}handra homomorphism.
\newblock {\em Invent. Math.}, 147(2):243--348, 2002.

\bibitem{CMFix}
M.~Finkelberg and V.~Ginzburg.
\newblock Calogero-{M}oser space and {K}ostka polynomials.
\newblock {\em Adv. Math.}, 172(1):137--150, 2002.

\bibitem{FultonYOungTableaux}
W.~Fulton.
\newblock {\em Young tableaux}, volume~35 of {\em London Mathematical Society
  Student Texts}.
\newblock Cambridge University Press, Cambridge, 1997.
\newblock With applications to representation theory and geometry.

\bibitem{Baby}
I.~G. Gordon.
\newblock Baby {V}erma modules for rational {C}herednik algebras.
\newblock {\em Bull. London Math. Soc.}, 35(3):321--336, 2003.

\bibitem{GriffethJack}
S.~Griffeth.
\newblock Orthogonal functions generalizing {J}ack polynomials.
\newblock {\em Trans. Amer. Math. Soc.}, 362(11):6131--6157, 2010.

\bibitem{HuybrechtsLehn}
D.~Huybrechts and M.~Lehn.
\newblock {\em The geometry of moduli spaces of sheaves}.
\newblock Cambridge Mathematical Library. Cambridge University Press,
  Cambridge, second edition, 2010.

\bibitem{KleshBook}
A.~Kleshchev.
\newblock {\em Linear and projective representations of symmetric groups},
  volume 163 of {\em Cambridge Tracts in Mathematics}.
\newblock Cambridge University Press, Cambridge, 2005.

\bibitem{MoDegen}
M.~Martino.
\newblock {B}locks of restricted rational {C}herednik algebras for
  {$G(m,d,n)$}.
\newblock {\em arXiv}, 1009.3200v1, 2010.

\bibitem{Miwa}
T.~Miwa, M.~Jimbo, and E.~Date.
\newblock {\em Solitons}, volume 135 of {\em Cambridge Tracts in Mathematics}.
\newblock Cambridge University Press, Cambridge, 2000.
\newblock Differential equations, symmetries and infinite-dimensional algebras,
  Translated from the 1993 Japanese original by Miles Reid.

\bibitem{BetheCherednik}
E.~Mukhin, V.~Tarasov, and A.~Varchenko.
\newblock Bethe algebra, {C}alogero-{M}oser space and {C}herednik algebra.
\newblock {\em arXiv}, 0906.5185v1, 2009.

\bibitem{SchubertGLN}
E.~Mukhin, V.~Tarasov, and A.~Varchenko.
\newblock Schubert calculus and representations of the general linear group.
\newblock {\em J. Amer. Math. Soc.}, 22(4):909--940, 2009.

\bibitem{KZChar}
E.~Mukhin, V.~Tarasov, and A.~Varchenko.
\newblock {KZ} characteristic variety as the zero set of classical
  {C}alogero-{M}oser {H}amiltonians.
\newblock {\em arXiv}, 1201.3990v2, 2012.

\bibitem{MTVDuality}
E.~E. Mukhin, V.~O. Tarasov, and A.~N. Varchenko.
\newblock Bispectral and {$(\mathfrak{gl}_N,\mathfrak{gl}_M)$} dualities.
\newblock {\em Funct. Anal. Other Math.}, 1(1):47--69, 2006.

\bibitem{WilsonSegalLoops}
G.~Segal and G.~Wilson.
\newblock Loop groups and equations of {K}d{V} type.
\newblock {\em Inst. Hautes \'Etudes Sci. Publ. Math.}, (61):5--65, 1985.

\bibitem{WilsonBi}
G.~Wilson.
\newblock Bispectral commutative ordinary differential operators.
\newblock {\em J. Reine Angew. Math.}, 442:177--204, 1993.

\bibitem{Wilson}
G.~Wilson.
\newblock Collisions of {C}alogero-{M}oser particles and an adelic
  {G}rassmannian.
\newblock {\em Invent. Math.}, 133(1):1--41, 1998.
\newblock With an appendix by I. G. Macdonald.

\end{thebibliography}

\def\cprime{$'$} \def\cprime{$'$} \def\cprime{$'$} \def\cprime{$'$}
  \def\cprime{$'$} \def\cprime{$'$} \def\cprime{$'$} \def\cprime{$'$}
  \def\cprime{$'$} \def\cprime{$'$} \def\cprime{$'$} \def\cprime{$'$}
  \def\cprime{$'$}

\end{document}